\documentclass[a4paper,oneside,11pt]{article}
\usepackage{a4wide} %, dsfont}
\usepackage{epsfig}
\usepackage{comment}
\usepackage[colorlinks, linkcolor = black, citecolor = black, filecolor = black, urlcolor = blue]{hyperref} 
\usepackage{xcolor}
\usepackage{graphicx}
\usepackage{natbib}
\usepackage[utf8]{inputenc}
\usepackage{amsmath}%
\usepackage{subcaption}%
\usepackage{url}
\usepackage{amssymb}%
\usepackage{amsthm}%
\usepackage{color}
\usepackage{verbatim}
\theoremstyle{plain}
\newtheorem{theorem}{Theorem}%
\newtheorem{lemma}{Lemma}%
\theoremstyle{definition}
\newtheorem{example}{Example}%
\newtheorem{remark}{Remark}%
\newtheorem{assumption}{Assumption}%

\renewcommand\theta{\vartheta}%wir brauchen kein \theta, ansonsten auskommentieren und im Text ÃÂÃÂÃÂÃÂ¤ndern
\renewcommand\epsilon{\varepsilon}
\newcommand{\N}{\ensuremath{\mathbb{N}}}%
\newcommand{\Z}{\ensuremath{\mathbb{Z}}}%
\newcommand{\R}{\ensuremath{\mathbb{R}}}%
\newcommand{\F}{\ensuremath{\mathcal{F}}}%
\newcommand{\mom}{\ensuremath{\mathrm{M}}}% # existing moments

\renewcommand\theta{\vartheta}%wir brauchen kein \theta, ansonsten auskommentieren und im Text ÃÂÃÂÃÂÃÂ¤ndern

\usepackage{natbib}
\begin{document}
%1
%\renewcommand{\baselinestretch}{1.75}
%
%%%%%%%%%%%%%%%%%%%%%%%%%%%%%%%%%%%%%%%%%%%%%%%%%%%%%%%%%%%%%%%%%%%%%%%%%%%%%%%
%%%%%%%%%%%%%%%%       Titel   %%%%%%%%%%%%%%%%%%%%%%%%%%%%%%%%%%%%%%%%%%%%%%%%
%%%%%%%%%%%%%%%%%%%%%%%%%%%%%%%%%%%%%%%%%%%%%%%%%%%%%%%%%%%%%%%%%%%%%%%%%%%%%%%
%
\LARGE
\vspace*{0.5cm}
\begin{center}
{\bfseries Simultaneous inference for   Berkson errors-in-variables regression  under fixed design
\\*[0.4cm]}
%
%\normalsize
\large

\bigskip
\Large
Katharina Proksch\footnote{Corresponding author: Dr.~Katharina Proksch, 
	Department of Applied Mathematics, University of Twente, Enschede, The Netherlands, 
	Email: k.proksch@utwente.nl}, Nicolai Bissantz$^2$ and 
 Hajo Holzmann$^3$\\*[0.1cm]
\large
$^1$University of Twente, The Netherlands\\
 Department of Applied Mathematics\\*[0.2cm]
$^2$Fakult{\"a}t f{\"u}r Mathematik\\
Ruhr-Universit{\"a}t Bochum, Germany\\*[0.2cm]
$^3$Fachbereich Mathematik und Informatik\\
Philipps-Universit{\"a}t Marburg, Germany\\*[0.2cm]
\today

\end{center}
\normalsize
\begin{abstract}
	In various applications of regression analysis, in addition to errors in the dependent observations also  errors in the predictor variables 
	play a substantial role and need to be incorporated in the statistical modeling process.
	In this paper we consider a nonparametric measurement error model of Berkson type with fixed design regressors and centered random errors, which is in contrast to much existing work in which the predictors are taken as random observations with random noise. 
	Based on an estimator that takes the error in the predictor into account and on a suitable Gaussian approximation, we derive 
	%uniform confidence statements for the function of interest. In particular, we provide 
	finite sample bounds on the coverage error of uniform confidence bands, where we circumvent the use of extreme-value theory
	and rather rely on recent results on anti-concentration of Gaussian processes.
	In a simulation study we investigate the performance of the uniform confidence sets for finite samples.
	\end{abstract}


\section{Introduction}
\label{intro}
In mean regression problems a predictor variable $X$, either a fixed design point or a random observation, is used to explain a response variable $Y$ in terms of the conditional mean regression function  $g(x)=\mathbb{E}[Y|X=x]$. The case of a random covariate occurs when both  $X$ and $Y$ are measured during an experiment and the case of fixed design corresponds to situations in which covariates can be set by the experimenter
such as a machine setting, say, in a physical or engineering experiment.
Writing $\epsilon = Y - \mathbb{E}[Y|X]$ gives the standard form of the non-parametric regression model 
$
Y=g(X)+\epsilon,
$
that is, the response is observed with an additional error but the predictor can be set or measured error-free. In many experimental settings this is not a suitable model assumption since either the predictor can also not be measured precisely, or since the presumed setting of the predictor does not correspond exactly to its actual value. There are subtle differences between these two cases, which we illustrate by the example of drill core measurements of the content
of climate gases in the polar ice. Assume that the content of climate gas $Y$ at the bottom of a drill hole is quantified. The depth of the drill hole $X$ is measured independently with error $\Delta$ giving the observation  $W$.
A corresponding regression model is of the following form
\begin{align}\label{classic}
	Y = g(X) + \epsilon, \qquad W=X +\Delta,
\end{align}
where  $W$, $\Delta$ and $\epsilon$ are independent,  $\Delta$ and $\epsilon$ are centered, and observations of $(Y, W)$ are available.
This model is often referred to as \textit{classical errors-in-variables model}.
A change in the experimental set-up might require a change in the model that is imposed. Assume that in our drill core experiment we fix specific depths $w$ at which the drill core is to be analyzed. However, due to imprecisions of the instrument we cannot accurately
fix the desired value of $w$, rather the true (but unknown) depth where the measurement is acquired is $w+\Delta$. In this case a corresponding model, referred to as \textit{Berkson errors-in-variables model} \citep{Berkson1950},  is of the form
\begin{align}\label{Berkson}
	Y = g(w+\Delta) + \epsilon, 
\end{align} 
where  $\Delta$ and $\epsilon$ are independent and centered, $w$ is set by the experimenter and $Y$ is observed.
In this paper we construct uniform confidence bands in the non-parametric Berkson errors-in-variables model with fixed design \eqref{Berkson}. In particular, we provide finite sample bounds on the coverage error of these bands. We also address the question how to choose the grid when approximating the supremum of a Gaussian process on $[0,1]$.
For Berkson-type measurement errors, a fixed design as considered in the present paper seems to be of particular relevance in experimentation in physics and engineering. 
Instead of using the classical approach based on results from extreme-value theory \citep{bicros1973a}, we propose a multiplier bootstrap procedure and construct asymptotic uniform confidence regions by using anti-concentration properties of  Gaussian processes which were recently derived by \cite{CheCheKat2014}.

	For an early, related contribution see \citet{neupol1998}, who develop the wild bootstrap originally proposed by \citet{wu1986jackknife} to construct confidence bands in a nonparametric heteroscedastic regression model with irregular design. While their method could potentially also be adopted in our setting, we preferred to work with the multiplier bootstrap which allows for a more transparent analysis.   
\\ 
%Our arguments do not involve a limit kernel of the deconvolution kernel as in  \citet{bisdumholmun2007, birbishol2010, prokschbissantzdette2012}.  As a consequence, the technical details are quite different. 
%
%As pointed out by many authors (see, e.g. \cite{DelHalQiu2006}), the assumption of an exactly known error-density $f_{\Delta}$ for identification is a standard assumption in the non-parametric treatment of errors-in-variables problems.
%However, \cite{SchHu2013} show that, in the classical, non-parametric errors-in-variables model, identifiability is guaranteed from the knowledge of data on the regressor and the dependent variable alone except for special cases which are specified.
\indent There is a vast literature on errors-in-variables models, where most of the earlier work is focused on parametric models \citep{Berkson1950, Anderson1984,Stefanski1985,Fuller1987}. A more recent overview of different models and methods can be found in the monograph by \cite{CarRupSteCra2006}.
In a non-parametric regression context, \cite{FanTru1993} consider the classical errors-in-variables setting \eqref{classic}, construct a kernel-type deconvolution estimator and investigate its asymptotic performance with respect to  weighted $L_p$-losses and $L_{\infty}$-loss and show rate-optimality for both ordinary smooth  and super smooth known distributions of errors $\Delta.$ The case of Berkson errors-in-variables with random design  is treated, e.\,g., in \cite{DelHalQiu2006}, who also assume a known error distribution,  \cite{Wang1993}, who assumes a parametric form of the error density, and \cite{Schennach2013}, whose method relies on the availability of an instrumental variable instead of the full knowledge of the error distribution. Furthermore, \cite{DelHalMei2008} consider the case in which the error-distribution is unknown but repeated measurements are available to estimate the error distribution. 
A mixture of both types of errors-in-variables is considered in \cite{CarDelHal2007} and the estimation of the observation-error variance is studied in \cite{DelHal2011}.\\
\indent However, in the aforementioned papers the focus is on estimation techniques and the investigation of  theoretical as well as numerical  performance of the estimators under consideration. In the non-parametric setting only very little can be found about the construction of statistical tests or confidence statements.
Model checks in the Berkson measurement error model are developed in \cite{KouSon2008,KouSon2009}, who construct goodness-of-fit tests for a parametric point hypothesis based on an empirical process approach and on a minimum-distance principle for estimating the regression function, respectively. 
The construction of confidence statements seems to be discussed only for classical errors in variables models with random design in \cite{DelHalJam2015}, who focus on pointwise confidence bands based on bootstrap methods and in \citet{kato2019uniform}, who provide uniform confidence bands.\\
\indent This paper is organized as follows. In Section \ref{sec:Pre} we discuss the mathematical details of our model and describe non-parametric methods for estimating the regression function in the fixed design Berkson model.  In Section \ref{sec:tech} we state the main theoretical results and in particular discuss the construction of confidence bands in Section \ref{sec:algorithms}, where we also discuss the choice of the bandwidth. The numerical performance of the proposed confidence bands is investigated in Section \ref{sec:sim}. Section \ref{sec:extensions} outlines an extension to error densities for which the Fourier transform is allowed to oscillate. %, and provides further numerical results.
Some auxiliary lemmas are stated in Section \ref{sec:Aux}. Technical proofs of the main results from Section \ref{sec:tech} are provided in Section \ref{Sec:Proofs}, while details and proofs for the extension in Section \ref{sec:extensions} along with some additional technical details are given in the Appendix, Section \ref{sec:appendix}. 
In the following, for a function $f$, which is bounded on some given interval $[a,b]$, we denote by $\|f\| = \|f\|_{[a,b]} = \sup_{x \in [a,b]}|f(x)|$ its supremum norm. The $L^p$-norm of $f$ over all of $\R$ is denoted by $\|f\|_p$. Further, for $w \in \R$ we set $\langle w\rangle:=(1+w^2)^{\frac{1}{2}}$.

\section{The Berkson errors-in-variables model with fixed design}\label{sec:Pre}
The Berkson errors-in-variables model with fixed design that we shall consider is given by
\begin{align}\label{model}
	Y_j = g(w_j + \Delta_j) + \epsilon_j,
\end{align}
where $ w_j = j/(n\,a_n)$, $j=-n, \ldots, n$, are the design points on a regular grid, $a_n$ is a design parameter that satisfies $a_n \to 0$, $n a_n \to \infty$, and $\Delta_j$ and $\epsilon_j$ are unobserved, centered, independent and identically distributed errors for which ${\rm Var} [\epsilon_1] = \sigma^2>0$ and $\mathbb{E}|\varepsilon_1|^{\mom}<\infty$ for some $\mom>2$. The density $f_\Delta$ of the errors $\Delta_j$ is assumed to be known. For ease of notation, we consider an equally spaced grid of design points here. However, this somewhat restrictive assumption can be relaxed to more general designs with a mild technical effort, as we elaborate in the Appendix, Section \ref{sec:nonequidist}. For random design Berkson errors-in-variables models, \cite{DelHalQiu2006} point out that identification of $g$ on a given interval requires an infinitely supported  design density if the error density  is of infinite support. This corresponds to our assumption that asymptotically, the fixed design exhausts the whole real line, which is assured by the requirements on the design parameter $a_n$. \cite{Meister2010} considers the particular case of normally distributed errors $\Delta$ and bounded design density, where a reconstruction of $g$ is possible by using an analytic extension.  
%
%It is clear that the asymptotic assumptions on the design can only serve as an approximation of reality with the intention to understand the properties of the estimator {\sl if we could} increase the number of data points unlimitedly. To this end the assumptions are reasonable and yield respectable results (see section \ref{Sec:Sim}).
%
If we define $\gamma$ as the convolution of $g$ and $f_{\Delta}(-\cdot),$ that is,
\[
\gamma(w)=\int_{\R} g(z)f_{\Delta}(z-w)\,dz,
\]
then $\mathbb{E}[ Y_j] = \gamma(w_j)$, and the calibrated regression model \citep{CarRupSteCra2006} associated with \eqref{model} is given by 
\begin{align}\label{modeleta}
	Y_j = \gamma(w_j) + \eta_j, \qquad \eta_j = g(w_j + \Delta_j) - \gamma(w_j) + \epsilon_j.
\end{align}
Here the errors $\eta_j$ are independent and centered as well but no longer identically distributed since their variances $\nu^2(w_j) = \mathbb{E}[ \eta_j^2 ]$ depend on the design points. To be precise, we have that
\begin{align}\label{VarEta}
	\nu^2(w_j) = \int \big(g(w_j + \delta) - \gamma(w_j) \big)^2\, f_\Delta(\delta)\, d\, \delta + \sigma^2\geq \sigma^2>0.
\end{align}  
This reveals the increased variability due to the errors in the predictors. 
\begin{figure}[ht]
	\begin{center}
		\includegraphics[width=0.6\textwidth]{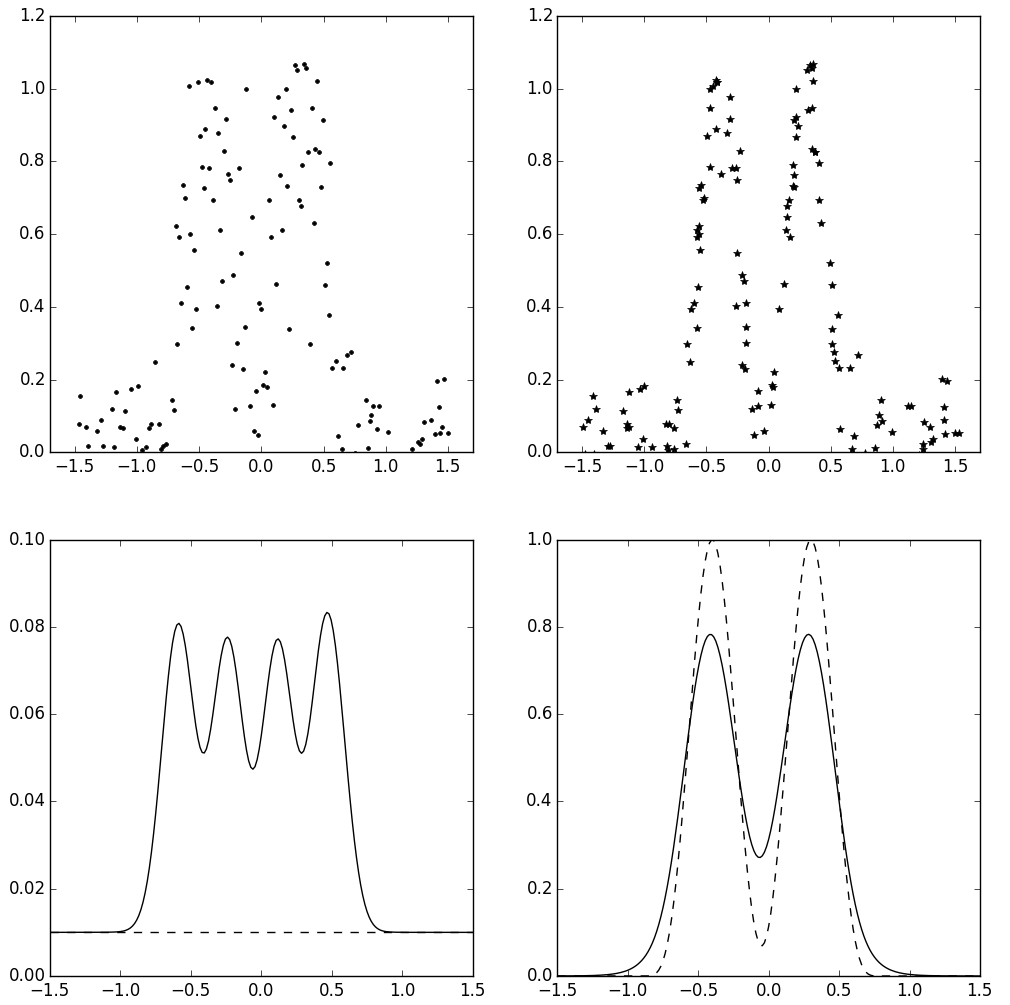}
	\end{center}
	\caption{\label{fig:ga} Alleged data points $(w_j, Y_j)$, actual data points $(w_j+\Delta_j,Y_j)$, a comparison between $g$ (dashed line) and  $\gamma$ (solid line) and a 
		comparison between $\sigma^2$ (dashed line) and $\nu^2$ (solid line) (clockwise from upper left to lower left).}
\end{figure}
The following considerations show that ignoring the errors in variables can lead to misinterpretations of the data at hand. To illustrate, in the setting of simulation Section \ref{sec:sim}, scenario 2, Figure \ref{fig:ga} (upper left panel) shows the alleged data points $(w_j, Y_j)$, that is, the observations at the incorrect, presumed positions, for a sample of size $n=100$. In addition to the usual variation introduced by the errors $\varepsilon_i$ in $y$-direction this display shows a variation in $x$-direction introduced by the errors in the $w_j$.  The upper right panel shows the actual but unobserved data points  $(w_j+\Delta_j,Y_j)$ that only contain the  variation in the $y$-direction.  Ignoring the errors-in-variables leads to estimating $\gamma$ instead of $g$, which introduces a systematic  error. The functions $\gamma$ (solid line) and $g$ (dashed line)  are both shown in the lower right panel of Figure \ref{fig:ga}.
The corresponding variance function is shown in the lower left panel of Figure \ref{fig:ga} (solid line) in comparison to the constant variance $\sigma^2$ (dotted line). 
Apparently, there is a close connection between the calibrated model \eqref{modeleta} and the classical deconvolution regression model as considered in \cite{birbishol2010} and \cite{prokschbissantzdette2012} in univariate and multivariate settings, respectively. In contrast to the calibrated regression model \eqref{modeleta}, in both works an i.i.d.~error structure is assumed.
Also, our theory provides finite sample bounds and is derived under weaker assumptions, requiring different techniques of proof. In particular, the  previous, asymptotic, results are derived under a stronger assumption on the convolution function $f_{\Delta}$.%, which allows to show the existence of a limit kernel for the deconvolution kernel. This assumption is dropped in this work.
To estimate $g$, we estimate the Fourier transform of $\gamma$, 
\[ \Phi_\gamma (t) = \int e^{i t w} \gamma(w)\, dw, \quad \text{by} \quad \widehat\Phi_\gamma(t) = \frac{1}{na_n}\sum\limits_{j=-n}^n Y_j e^{it w_j}.\]
An estimator for $g$ is then given by
\begin{equation}\label{eq:theestimator}
	\hat g_n (x; h) = \frac{1}{2\pi}\int_{\R}  e^{-i t x} \Phi_k(h t)\frac{\hat\Phi_\gamma (t )}{\Phi_{f_\Delta}(-t)}dt.
	%= \frac{1}{2\pi} \int_{\R}  e^{-i t x} \Phi_k(h t)\frac{\hat\Phi_\gamma (t )}{\overline{\Phi}_{f_\Delta}(t)}dt.
\end{equation}
Here $h>0$ is a smoothing parameter called the bandwidth, and $\Phi_k$ is the Fourier transform of a bandlimited kernel function $k$ that satisfies Assumption \ref{AssKernel} below. Notice that both $\Phi_{\gamma}$ and $\Phi_{f_{\Delta}}$ tend to zero as $|t|\to\infty$ such that estimation of $\Phi_{\gamma}$ in \eqref{eq:theestimator} introduces instabilities for large values of $|t|$. Since the kernel $k$ is bandlimited, the function $\Phi_k$ is compactly supported and the factor $\Phi_k(ht)$ discards large values of $t$, therefore serving as regularization. The estimator can be rewritten in kernel form as follows:
\[
\hat g_n(x;h) = \frac{1}{na_nh } \sum\limits_{j=-n}^n Y_j K\left(\frac{w_j-x}{h};h\right), 
\]
where the deconvolution kernel $K(\cdot;h)$ is given by
\begin{equation}\label{eq:asympkernel}
	K(w;h) = \frac{1}{2\pi}\int_{\R}  e^{-it w} \frac{\Phi_k(t)}{\Phi_{f_{\Delta}}(-t/h)} dt.
\end{equation}
\section{Theory}\label{sec:tech}
By $\mathcal{W}^m(\R)$ we denote the Sobolev spaces $\mathcal{W}^m(\R)=\{g\,|\,\|\Phi_g(\cdot) \,\langle\,\cdot\,\rangle^m \|_2<\infty\}$, $m>0$, where we recall that $\langle w\rangle:=(1+w^2)^{\frac{1}{2}}$ for $w \in \R$. 
We shall require the following assumptions. 
\begin{assumption}\label{AssSmoothness}
	The functions $g$ and $f_\Delta$ satisfy
	\begin{itemize}
		\item[(i)] $g\in\mathcal{W}^m(\R)\cap L^{r}(\R)\quad\text{ for all}\quad r\leq\mom$ and for some $m>5/2$,
		\item[(ii)] $ f_{\Delta}$ is a bounded, continuous, square-integrable density,
		\item[(iii)] $ \Phi_{g*f_{\Delta}}=\Phi_g\cdot\Phi_{f_{\Delta}}\in\mathcal{W}^{s}(\R)
		$ for some $s>1/2$.
	\end{itemize}
\end{assumption}
	Assumption \ref{AssSmoothness} (i) stated above is a smoothness assumption on the function $g$. In Lemma \ref{Le:g} in Section \ref{sub:propnu} we list the properties of $g$ that are frequently used throughout this paper and that are implied by this assumption. In particular, by Sobolev embedding, $m>5/2$ implies that the function $g$ is twice continuously differentiable, which is used in the proof of Lemma \ref{Lemma:BiasVarianz}.
\begin{assumption}\label{AssKernel}
	Let $\Phi_k\in C^2(\R)$ be symmetric, $\Phi_k(t)\equiv1$ for all $t\in[-D,D],$ $0<D<1,$ $|\Phi_k(t)|\leq1$ and $\Phi_k(t)=0,\,|t|>1.$
\end{assumption}
In contrast to kernel-estimators in a classical non-parametric regression context, the kernel $K(\cdot;h)$, defined in \eqref{eq:asympkernel}, depends on the bandwidth $h$ and hence on the sample size via the factor $1/\Phi_{f_{\Delta}}(-t/h)$. For this reason, the asymptotic behavior of $K(\cdot;h)$ is determined by the properties of the Fourier transform of the error-density $f_{\Delta}$. The following assumption on $\Phi_{f_{\Delta}}$ is standard in the non-parametric deconvolution context \citep[see, e.g.,][]{kato2019uniform, schmidt2013multiscale} and will be relaxed in Section \ref{sec:extensions} below.
\begin{assumption}\label{AssConvergenceS}
	Assume that $\Phi_{ f_{\Delta}}(t)\neq0$ for all $t\in\R$ and that there exist constants $\beta>0$ and $0<c<C$, $0<C_S$ such that
	\begin{align}\label{S}
		c\langle t\rangle^{-\beta}\leq|\Phi_{ f_{\Delta}}(t)|\leq C \langle t\rangle^{-\beta}\quad\text{and}\quad\bigl|\Phi_{ f_{\Delta}}^{(1)}(t)\bigr|\leq C_S\langle t\rangle^{-\beta-1}. \tag{S}
	\end{align}
\end{assumption}

A standard example of a density that satisfies Assumption \ref{AssConvergenceS} is the Laplace density with parameter $a>0$,
\begin{align}\label{Laplace}
	f_{\Delta,0}(a;x)=\tfrac{a}{2}e^{-a|x|}\quad\text{with}\quad\Phi_{ f_{\Delta,0}}(a;t)=\langle t/a\rangle^{-2}.
\end{align}
In this case we find $\beta=2$, $C=a^2\vee1$, $c=a^2\wedge1$ and $C_S=2/a^2\vee 2a^2.$
%
%\section{Simultaneous inference for the Berkson errors-in-variables model}
	\begin{remark}
		Our asymptotic theory cannot accommodate the case of exponential decay of the Fourier transform of the density $f_{\Delta}$, as the asymptotic behaviour of the estimators in the supersmooth case and the ordinary smooth case differs drastically. While for the ordinary smooth case considered here $\hat g(x)$ and $\hat g(y)$ are asymptotically independent if $x\neq y$, convolution with a supersmooth distribution is no longer local and causes dependencies throughout the domain. This leads to different properties of the suprema $\sup_{x\in[0,1]}|\hat g(x;h)-\mathbb{E}[\hat g(x;h)]|, $ which play a crucial role in the construction of our confidence bands.   
		In particular, the asymptotics strongly depend on the exact decay of the characteristic function $\Phi_{f_{\Delta}}$ and needs a treatment on a case to case basis (more details on the latter issue can be found in  \cite{vEandG2008})). 
	\end{remark}

\subsection{Simultaneous inference}\label{sec:GaussAppr}
Our main goal is to derive a method to conduct uniform inference on the regression function $g$, which is based on a Gaussian approximation to the maximal deviation of $\hat g_n$ from $g$. We consider the usual decomposition of the difference $g(x)-\hat g_n(x;h)$ into deterministic and stochastic parts, that is
\begin{align*}
	g(x)-\hat g_n(x;h)=g(x)-\mathbb E[\hat g_n(x;h)]+\mathbb E[\hat g_n(x;h)]-\hat g_n(x;h),
\end{align*}
where
%
%\[ \mathbb E [\hat g_n (x; h)] = \frac{1}{na_nh } \sum\limits_{j=-n}^n \gamma(w_j)\, K\left(\frac{w_j-x}{h};h\right),\]
%
% and%
\begin{align}\label{start}
	\hat g_n(x;h) - \mathbb E[\hat g_n(x;h)]  = \frac{1}{na_nh} \sum\limits_{j=-n}^n \eta_j K\left(\frac{w_j-x}{h};h\right).
\end{align}
If the bias, the rate of convergence of which is given in Lemma \ref{Lemma:BiasVarianz} in Section \ref{sec:Aux}, is taken care of by choosing an undersmoothing bandwidth $h$, the stochastic term \eqref{start} in the above decomposition dominates.

Theorem \ref{GaussAppr} below is the basic ingredient for the construction of the confidence statements under Assumption \ref{AssConvergenceS}. It guarantees that the random sum \eqref{start} can be approximated by a distribution free Gaussian version, uniformly with respect to $x\in[0,1]$, that is, a weighted sum of independent, normally distributed random variables such that the required quantiles can be estimated from this approximation. In the following assumption conditions on the bandwidth and the design parameter are listed which will be needed for the theoretical results.
\begin{assumption}\label{AssBandwidth} ~\\
	\vspace{-0.6cm}	
	\begin{enumerate}
		\item[(i)]	$
		\ln(n)n^{\frac{2}{\mom}-1}/(a_nh)+\frac{h}{a_n}+\ln(n)^2h+\ln(n)a_n+\frac{1}{na_nh^{1+2\beta}}=o(1)$,
		\item[(ii)] $\sqrt{na_nh^{2m+2\beta}}+\sqrt{na_n^{2s+1}h^2}+1/\sqrt{na_nh^2} =o(1/\sqrt{\ln(n)})$.
	\end{enumerate}	
\end{assumption}
The following example is a short version of a lengthy discussion given in the Appendix, Section \ref{sec:hyperrole}. More details can be found there.	

	\begin{example}\label{ex:bw}
		In a typical setting, the conditions listed in Assumption \ref{AssBandwidth} are satisfied if $h$ is the rate optimal bandwidth of classical deconvolution problems. As an example, consider the case of a function $g\in\mathcal{W}^m(\R), m>5/2,$ of bounded support, $f_{\Delta}$ as in \eqref{Laplace} and $\mathbb{E}[\epsilon_1^4]\leq\infty$. Then $\beta=2$ and Assumption \ref{AssSmoothness} (iii) holds for any $s>0$ such that $a_n$ can be chosen of order $n^{-\varepsilon}$ for $\varepsilon$ arbitrarily small. The rate optimal bandwidth in the classical deconvolution problem is of order $n^{-{1/(2(m\beta))}}$ \citep{Fan91}. With the choices of $a_n=n^{-\varepsilon}$ and $h=n^{-{1/(2(m\beta))}}$, $\varepsilon$ sufficiently small and $s$ sufficiently large, Assumption \ref{AssBandwidth} (i) and (ii) reduce to the requirements $1/(na_nh^{1+2\beta})=o(1)$ and $na_nh^{2m+2\beta}\ln(n)=o(1)$, respectively. These are met for small $\varepsilon>0$ since $1/(na_nh^{1+2\beta})\geq n^{-1+\varepsilon+5/9}$ and $na_nh^{2m+2\beta}\ln(n)=\ln(n)n^{-\varepsilon}$.
\end{example}
The first term in Assumption \ref{AssBandwidth} (i) stems from the Gaussian approximation and becomes less restrictive if the number of existing moments of the errors $\varepsilon_i$ increases. The last term in (i) guarantees that the variance of the estimator tends to zero. The terms in between are only weak requirements and are needed for the estimation of certain integrals. Assumption \ref{AssBandwidth}  (ii) guarantees that the bias is negligible under Assumption \ref{AssConvergenceS}. The first term  guarantees undersmoothing, the second term stems from the fact that only observations from the finite grid $[-1/a_n, 1/a_n]$ are available, while the third term accounts for the discretization bias. It is no additional restriction if $\beta>1/2$. For a given interval $[a,b]$, recall that $\|f\| = \|f\|_{[a,b]}$ denotes the supremum norm of a bounded function on $[a,b]$.
\begin{theorem}\label{GaussAppr}
	Let Assumptions \ref{AssKernel} - \ref{AssConvergenceS} and \ref{AssBandwidth} (i) be satisfied. For some given interval $[a,b]$ of interest, let  $\hat{\nu}_n$ be a nonparametric estimator of the standard deviation in model \eqref{modeleta} such that $\hat \nu>\sigma/2$ and
	\begin{align}\label{eq:convnu}
		\mathbb{P}\left(\left\| \tfrac{1}{\hat\nu}-\tfrac{1}{\nu}\right\|_{[a,b]} 
		%=  \| \tfrac{1}{\hat\nu}-\tfrac{1}{\nu}\|_{[a,b]} =o_{\mathbb{P}}(1/\sqrt{\ln(n)}).
		>\tfrac{n^{2/\mom}}{\sqrt{na_nh}}	\right)=o(1).
	\end{align}
	\begin{enumerate}
		\item There exists a sequence of independent standard normally distributed random variables $(Z_n)_{n\in\Z}$ such that for
		\begin{align}
			\mathbb{D}_n(x) & :=\frac{\sqrt{na_nh}h^{\beta}}{\hat{\nu}(x)}\bigl(\hat g_n(x;h)-\mathbb{E}[\hat g_n(x;h)]\bigr),\nonumber\\
			\mathbb{G}_{n}(x)& :=\frac{h^{\beta}}{\sqrt{na_nh}} \sum\limits_{j=-n}^n Z_{j} K\left(\tfrac{w_j-x}{h};h\right), \label{Gaussprozess}
		\end{align}
		we have that for all $\alpha\in(0,1)$
		%
		% \begin{align}\label{Th1}
		%	\big\|\mathbb{D}_n-\mathbb{G}_{n}\big\|=o_{\mathbb{P}}\big(\ln(n)^{-\frac{1}{2}}\big).
		% \end{align}

		%
		\begin{equation}\label{eq:quantileapprox}
			\left| \mathbb{P}\big(\|\mathbb{D}_{n}\|\leq q_{ \|\mathbb{G}_{n}\|}(\alpha) \big)-\alpha\right|\leq r_{n,1},
		\end{equation}
		where $q_{ \|\mathbb{G}_{n}\|}(\alpha)$ is the $\alpha$-quantile of  $\|\mathbb{G}_{n}\|$ and for some constant $C>0$
		\begin{align*}
			r_{n,1}=\mathbb{P}\left(\left\| \tfrac{1}{\hat\nu}-\tfrac{1}{\nu}\right\| 
			>\tfrac{n^{2/\mom}}{\sqrt{na_nh}}	\right)+C\left(\tfrac{1}{n}+\tfrac{n^{2/\mom}\sqrt{\ln(n)^3})}{\sqrt{na_nh}}\right).
		\end{align*} 
		\item If, in addition, Assumption \ref{AssBandwidth} (ii) and Assumption \ref{AssSmoothness} are satisfied, $\mathbb{E}[\hat g_n(x;h)]$ in \eqref{Gaussprozess} can be replaced by $g(x)$ with an additional error term of order $r_{n,2}=\sqrt{na_nh^{2m+2\beta}}+\sqrt{na_n^{2s+1}h^2}+1/\sqrt{na_nh^2}$.
		
	\end{enumerate}				
\end{theorem}

\noindent  In particular, Theorem \ref{GaussAppr} implies that $\lim_{n \to \infty} \mathbb{P}\big(\|\mathbb{D}_{n}\|\leq q_{ \|\mathbb{G}_{n}\|}(\alpha) \big)= \alpha$ for all $ \alpha \in (0,1)$.  Regarding assumption \eqref{eq:convnu}, properties of variance estimators in a heteroscedastic non-parametric regression model are discussed in \citet{brown}. 

The following theorem is concerned with suitable grid widths of discrete grids $\mathcal{X}_{n,m}\subset[a,b]$  such that the maximum over $[a,b]$ and the maximum over $\mathcal{X}_{n,m}$ behave asymptotically equivalently.
\begin{theorem}\label{Thm:Grid}
	For some given interval $[a,b]$ of interest, let $\mathcal{X}_{n,m}\subset[a,b]$ a grid of points $a=x_{0,n}\leq x_{1,n}\leq\ldots\leq x_{m,n}=b$. Let $\|f\|_{\mathcal{X}_{n,m}}:=\max_{x\in\mathcal{X}_{n,m}}|f(x)|$. If the grid is sufficiently fine, i.e.,
	$$
	|\mathcal{X}_{n,m}|:=\max_{1\leq i\leq m}|x_{i,m}-x_{i-1,m}|\leq \frac{h^{1/2}}{na_n^{1/2}},
	$$ 
	then, under the assumptions of Theorem \ref{GaussAppr}, The following holds.
	\begin{enumerate}
		\item For all $\alpha\in(0,1)$
		\begin{equation}\label{eq:approxdiscrete}
			\left|\mathbb{P}\big(\|\mathbb{D}_{n}\|\leq q_{ \|\mathbb{G}_{n}\|_{\mathcal{X}_{n,m}}}(\alpha) \big)- \alpha\right|\leq r_{n,1}(1+o(1)).
		\end{equation}
		\item If, in addition, Assumption \ref{AssBandwidth} (ii) and Assumption \ref{AssSmoothness} are satisfied, $\mathbb{E}[\hat g_n(x;h)]$ in \eqref{eq:approxdiscrete} can be replaced by $g(x)$ with an additional error term of order $r_{n,2}=\sqrt{na_nh^{2m+2\beta}}+\sqrt{na_n^{2s+1}h^2}+1/\sqrt{na_nh^2}$.
	\end{enumerate}
\end{theorem}

\subsection{Construction of the confidence sets and bandwidth choice}\label{sec:algorithms}

In this section we present an algorithm which can be used to construct  uniform confidence sets based on Theorem \ref{GaussAppr}. 
Let $\mathbb{G}_n(x)$ be the statistic defined in \eqref{Gaussprozess}.
In order to obtain quantiles that guarantee uniform coverage of a confidence band, generate $M$ times  $\|\mathbb{G}_n\|_{\mathcal{X}_{n,m}}$, where $|\mathcal{X}_{n,m}|=o(h^{3/2}a_n^{1/2}/\ln(n))$ (see Theorem \ref{Thm:Grid}), that is, calculate $\hat{\nu}_n(x)$ for $x\in\mathcal{X}_{n,m}$,  generate $M$ times $2n+1$ realizations of independent, standard normally distributed random variables $Z_{1,j},\ldots,Z_{2n+1,j},\;j=1,\ldots,M.$ Calculate 
$
\mathbf{M}_{n,j}:=\max_{x\in\mathcal{X}_{n,m}}|\mathbb{G}_{n,j}(x)|.
$
Estimate the $(1-\alpha)$-quantile of $\|\mathbb{G}_n\|$ from $\mathbf{M}_{n,1},\ldots, \mathbf{M}_{n,M}$ and denote the estimated quantile by $\hat{q}_{\|\mathbb{G}_n\|_{\mathcal{X}_{n,m}}}(1-\alpha).$
From Theorem \ref{GaussAppr} we obtain the confidence band

\begin{align}\label{eq:confband}
	\hat{g}_n(x;h) \, \pm \,  \hat{q}_{\|\mathbb{G}_n\|_{\mathcal{X}_{n,m}}}(1-\alpha)\frac{\hat{\nu}_n(x)}{\sqrt{na_n}h^{1/2+\beta}}, \qquad x\in[a,b].
\end{align}
	\begin{remark}	
		Given a suitable estimator for the variance $\nu^2$, Theorem \ref{GaussAppr} and Example \ref{ex:bw} imply that, typically, the coverage error of the above bands will be of order $n^{2/\mom}\sqrt{\ln(n)}/\sqrt{na_nh}+\sqrt{na_nh^{2m+2\beta}}$.  The first term is determined by the accuracy of the Gaussian approximation and will be negligible if the distribution of the errors $\epsilon_i$ possesses sufficiently many moments, while the second term is of order $\sqrt{a_n}$ if the optimal bandwidth of classical deconvolution problems is used. This shows that, in contrast to confidence bands based on asymptotic quantiles, the coverage error typically decays polynomially in $n$. 
	\end{remark} 
	\begin{remark}
		In nonparametric regression without errors-in-variables the widths of uniform confidence bands are of order $\sqrt{\ln(n)}/\sqrt{nh}$ \citep[see, e.g.,][]{neupol1998}. Our bands \eqref{eq:confband} are wider by the factor $1/(a_nh^{\beta})$ which is due to the ill-posedness ($\beta$) and the, possibly slow, decay of $\gamma$ (expressed in terms of $a_n$).
	\end{remark}		

	For the choice of the bandwidth, \citet{gine2010confidence} (see also \citep{CheCheKat2014}) convincingly demonstrated how to use Lepski's method to adapt to unknown smoothness when constructing confidence bands. 
	In our framework, choose an exponential grid of bandwidths $h_k = 2^{-k}$ for $k \in \{k_l, \ldots, k_u\}$, with $k_l, k_u \in \N$ being such that $2^{- k_u } \simeq 1/n$ and \linebreak $2^{-k_l} \simeq \big( (\log n) / (n a_n) \big)^{1/(\beta + \bar m)}$ and where $\bar m$ corresponds to the maximal degree of smoothness to which one intends to adapt. Then for a sufficiently large constant $C_L>0 $ choose the index $k$ according to 
	\begin{align*}
		\hat k & = \min\big\{ k \in \{k_l, \ldots, k_u\}  \mid \| \hat g(\cdot ; h_k) - \hat g(\cdot ; h_l)\| \leq \, C_L\, \Big(\frac{\log n}{n\, a_n\, h^{1+2\,\beta}_l } \Big)^{1/2}\\
		&\hspace{6cm} \forall\ k \leq l \leq k_u \big\},
	\end{align*}
	and choose an undersmoothing bandwidth according as $\hat h = h_{\hat k}/\log n$. A result in analogy to \citet{gine2010confidence} would imply that under an additional \textit{self-similarity} condition on the regression function $g$, using $\hat h$ in \eqref{eq:confband} produces confidence bands of width $\big(\log n / (n\, a_n)\big)^{\frac{m-1/2}{\beta + m}}\, (\log n)^{\beta + 1/2}$ if $g$ has smoothness $m$. 
	Technicalities in our setting would be even more involved due to the truncated exhaustive design involving the parameter $a_n$. Therefore, we refrain from going into the technical details. In the subsequent simulations we use a simplified bandwidth selection rule which, however, resembles the Lepski method. 

\section{Simulations}\label{sec:sim}

\begin{table}[h]
	\begin{tabular}{|c|c|c||c|c||c|c|}
		\hline
		& $n=100$ & $n=100$ & $n=750$ & $n=750$\\
		& $\sigma=\sigma_\delta=0.1$ & $\sigma=\sigma_\delta=0.05$ & $\sigma=\sigma_\delta=0.1$ & $\sigma=\sigma_\delta=0.05$\\
		\hline
		\hline
		$g_a$ & $0.25$ & $0.24$ & $0.21$ & $0.12$ \\
		\hline
		$g_b$ & $0.20$ & $0.22$ & $0.22$ & $0.11$ \\ % Signal 1
		\hline
	\end{tabular}
	\caption{\label{table:bdw} Regularization parameter used in the subsequent simulations. See text for details on its selection.% from our regularization parameter selection procedure based on the visual evaluation 
		%		of a sequence of estimates for different regularization parameters in advance of the 
		%		confidence band simulations. }
	}
\end{table}

%Nyquist: $g_a$ & $0.12$ & $0.125$ & $0.019375$ & $0.03375$ \\ round(3/(0.12*100),2)
%\hline
%$g_b$ & $0.15$ & $0.1375$ & $0.018$ & $0.0375$ \\ % Signal 1

\begin{table}[h]
	\begin{tabular}{|c|c|c||c|c||c|c|}
		\hline
		& $n=100$ & $n=100$ & $n=750$ & $n=750$\\
		& $\sigma=\sigma_\delta=0.1$ & $\sigma=\sigma_\delta=0.05$ & $\sigma=\sigma_\delta=0.1$ & $\sigma=\sigma_\delta=0.05$\\
		\hline
		\hline
		$g_a$ & $5.8\%$ % 0.44
		& $7.2\%$ % 0.08
		& $5.1\%$ & % 0.21
		$5.6\%$ % 0.14-0.01
		\\
		\hline
		$g_b$ & $1.8\%$ % 0.86
		& $5.3\%$ & %0.28+0.01 
		$5.0\%$ % 0.24
		& $5.0\%$ \\ % 0.22
		\hline
		
	\end{tabular}
	\caption{\label{table:ok1} Simulated rejection probabilities for bootstrap confidence bands.}
\end{table}

\begin{table}[h]
	\begin{tabular}{|c|c|c||c|c||c|c|}
		\hline
		& $n=100$ & $n=100$ & $n=750$ & $n=750$\\
		& $\sigma=\sigma_\delta=0.1$ & $\sigma=\sigma_\delta=0.05$ & $\sigma=\sigma_\delta=0.1$ & $\sigma=\sigma_\delta=0.05$\\
		\hline
		\hline
		$g_a$  & $0.44$ & $0.16$ & $0.21$ & $0.14$ \\ 
		\hline
		$g_b$  & $0.86$ & $0.28$ & $0.24$ & $0.22$ \\
		\hline
		
	\end{tabular}
	\caption{\label{table:ok2} Average width of bootstrap confidence bands.}
\end{table}

\begin{figure}[ht]
	\begin{center}
		\includegraphics[scale=0.4]{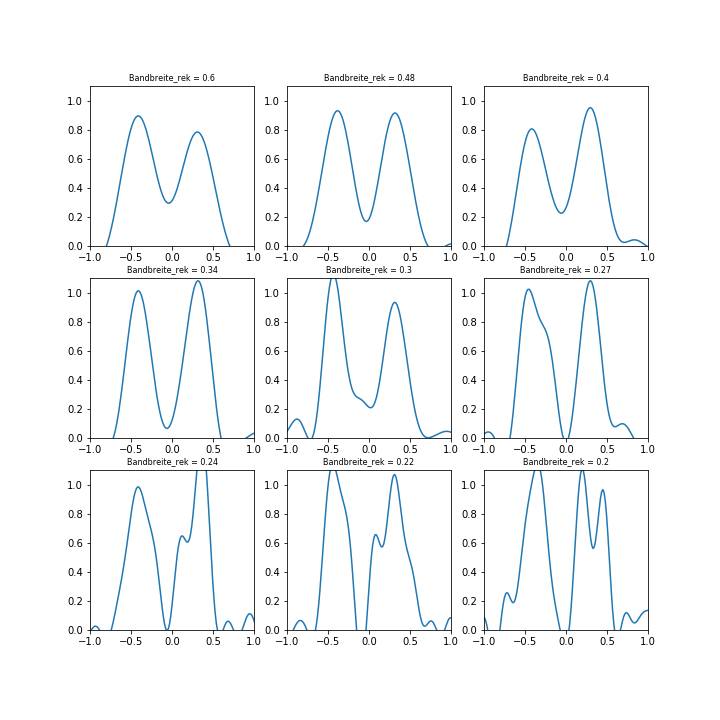}
	\end{center}
	\caption{\label{fig:estexamples} Sequence of estimates for increasing regularization parameter from a random sample of observations of signal $g_b$ with $n=100$ and $\sigma=\sigma_\delta=0.1$.}
\end{figure}

\begin{figure}[ht]
	\begin{center}
		\includegraphics[scale=0.4]{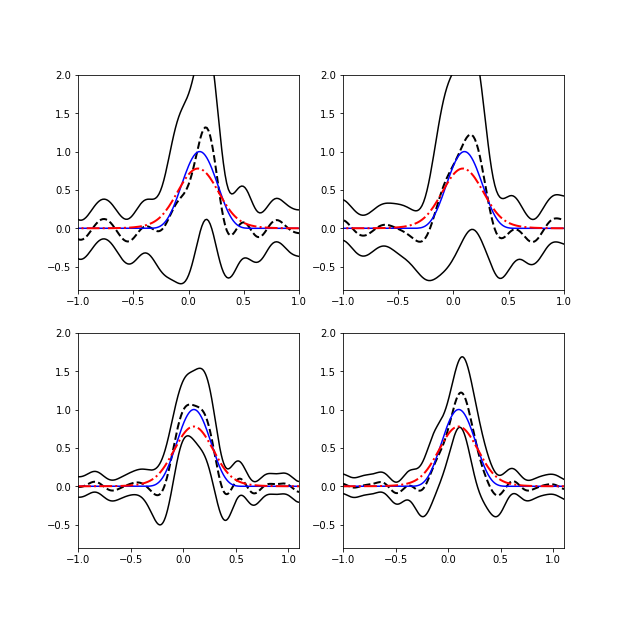}
	\end{center}
	\caption{\label{fig:kbga} True signal (solid line), observable signal (dash-dotted line) and estimates and associated confidence bands (dashed lines) from four random samples for $g_a$ for $n=100$ (top) and $n=750$ (bottom) and $\sigma=\sigma_\delta=0.1$.}
\end{figure}

\begin{figure}[ht]
	\begin{center}
		\includegraphics[scale=0.4]{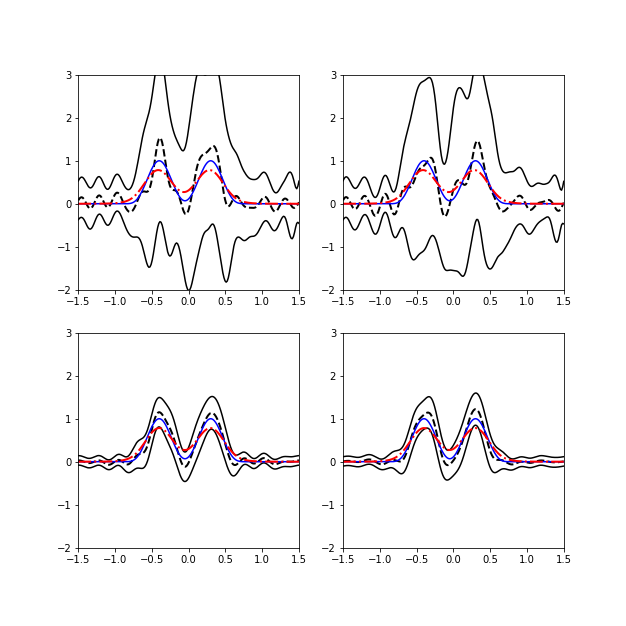}
	\end{center}
	\caption{\label{fig:kbgb} True signal (solid line), observable signal (dash-dotted line) and estimates and associated confidence bands (dashed lines) from four random samples  for $g_b$ for $n=100$ (top) and $n=750$ (bottom) and $\sigma=\sigma_\delta=0.1$.}
\end{figure}

%\begin{figure}[ht]
%	%	\begin{center}
%	\begin{subfigure}[c]{0.5\textwidth}
%		\includegraphics[width=0.8\textwidth]{berkson_signal1_n100_d1_e1_bild3_2.png}
%	\end{subfigure}
%	\begin{subfigure}[c]{0.5\textwidth}
%		\includegraphics[width=0.8\textwidth]{berkson_signal1_n100_d1_e1_bild3_4.png}
%	\end{subfigure}
%	%	\end{center}
%	\caption{\label{fig:rej} Simulated rejection probabilities (left) and confidence band sizes (right) for bootstrap confidence bands for $\sigma=\sigma_\delta=0.1$ and signal $g_b$ as functions of the regularization parameter. The size of the confidence band is measured by the area covered in the region of interest $[-0.7,\,0.6]$ 
%		(see text for details).}
%\end{figure}
%
In this section we investigate the numerical performance of our proposed methods in finite samples. We consider several different computational scenarios. As regression functions we consider
\begin{align*}
	g_a(x)=(1-4(x-0.1)^2)^5I_{[0,1]}(2|x-0.1|), % Modell 2 alt g_a
\end{align*}
and
\begin{align*}
	g_b(x)=(1-4(x+0.4)^2)^5I_{[0,1]}(2|x+0.4|)+(1-4(x-0.3)^2)^5I_{[0,1]}(2|x-0.3|). % Modell 1 alt g_b
\end{align*}
For the error distribution $f_{\Delta}$ we chose two densities of a Laplace distribution as defined in \eqref{Laplace}
with $a=\frac{0.1}{\sqrt{2}}$ and $a=\frac{0.05}{\sqrt{2}}$, i.e. standard deviations $\sigma_\delta=0.1$ and $\sigma_\delta=0.05$, respectively. Finally, 
$a_n=2/3$ in all simulations discussed below.
Our estimation is based on an application of the Fast Fourier transform implemented in python/scipy.
The integration used a damped version of a spectral cut off with cut-off function $I(\omega)=1-\exp(-\frac{1}{(\omega\cdot h)^{2}})$
in spectral space.\\
Construction of the confidence bands requires the selection of a regularization parameter for the estimator $\hat g$. In our simulations, we have chosen this parameter by
a visual inspection of a sequence of estimates for the regularization parameter, covering a range from over- to under-smoothing, see Figure \ref{fig:estexamples}.  We chose the minimal regularization parameter for which the estimates 
do not change systematically in overall amplitude, but appear to only exhibit additional random fluctuations 
at smaller values of the parameter. In the case shown here, we chose a regularization parameter of $0.27$. The same procedure was followed for 
other combinations of $n, \sigma, \sigma_\delta$ and signal  $g_a$ resp. $g_b$) and the results can be found
in Table \ref{table:bdw}. 
%Note that the regularization parameters shown in the table here are in multiples of $\frac{1}{2}$ the largest frequency for which the Fourier transform 
%can be estimated from data with the grid size used. Hence, it decreases with increasing sample size (but variance terms unchanged),
%since $a_n=2/3$ is a fixed value and the regularization parameter, which is optimal increases sublinear with $n$.
This regularization parameter was then kept fixed for each combination of $n,\sigma,\sigma_\delta$ and signal $g\in\{g_a,g_b\}$.
Figures \ref{fig:kbga} and \ref{fig:kbgb} show four random examples each for estimates of $g_a$ and $g_b$, respectively, together with the associated 
confidence bands from $250$ bootstrap simulations. Solid lines represent the true signal $g_a$ and $g_b$ and dashed lines the estimates $\hat g_n$ together with their 
associated confidence bands. Again, in both cases, $n=100$, $\sigma=0.1$ and $\sigma_\delta=0.1$.
Next, we discuss the practical performance of the bootstrap confidence bands in more detail for the first scenario, where the model is correctly specified and the errors 
in the predictors are taken into account as well. The results are shown in Tables \ref{table:ok1} and \ref{table:ok2} for the simulated rejection probabilities (one minus the coverage probability) at a nominal value of $5\%$ and for the (average) width of the confidence bands. 
In all cases, we performed simulations based on $500$ random samples of data and nominal rejection probability $5\%$ (i.e. confidence bands with nominal coverage
probability of $95\%$). For each of these data samples, we repeated $250$ times the following scenario. First, we determined the width of the confidence bands from $250$ bootstrap 
simulations and second, we evaluated whether the confidence bands cover the true signal everywhere in an interval of interest. The numbers shown in the table give the percentage of rejections, i.e. of 
where the confidence bands do not overlap the true signal everywhere in such an interval. Here, the intervals of interest are chosen as an interval where the respective signal 
is significantly different from $0$. The intention of this is that in many practical applications the data analyst is particularly interested in those parts of the signal. 
Here, we chose the interval $[-0.7,\,0.6]$ as 'interval of interest' for $g_a$ and $g_b$.
From the tables we conclude that the method performs well, particularly for $n=750$, where the confidence bands are substantially less wide.

\section{Extensions}\label{sec:extensions}
The following assumption is less restrictive than Assumption \ref{AssConvergenceS}, \eqref{S}. 
\begin{assumption}\label{AssConvergenceW}
	Assume that $\Phi_{ f_{\Delta}}(t)\neq0$ for all $t\in\R$ and that there exist constants $\beta>0$ and $0<c<C$, $0<C_W$ such that
	\begin{align}\label{W}
		c\langle t\rangle^{-\beta}\leq|\Phi_{ f_{\Delta}}(t)|\leq C \langle t\rangle^{-\beta}\quad\text{and}\quad\bigl|\Phi_{ f_{\Delta}}^{(1)}(t)\bigr|\leq C_W\langle t\rangle^{-\beta}.\tag{W}
	\end{align}
\end{assumption}

An example for a density that satisfies Assumption \ref{AssConvergenceW}	but not  Assumption \ref{AssConvergenceS} is given by the mixture
\begin{align}\label{eq:mixt}
	f_{\Delta,1}(1;x)=\frac{\lambda}{2}f_{\Delta,0}(1;x-\mu)+\frac{\lambda}{2}f_{\Delta,0}(1;x+\mu)+(1-\lambda)f_{\Delta,0}(1;x),
\end{align}
where $\lambda\in(0,1/2)$ and $\mu\neq0$, and $f_{\Delta,0}$ is the Laplace density defined in \eqref{Laplace}. We find
\begin{align*}
	\Phi_{f_{\Delta,1}}(t)=(1-\lambda+\lambda\cos(\mu t))\langle t\rangle^{-2},
\end{align*}
which yields $\beta=2$, $c=1-2\lambda$ and $C_W=\lambda\mu+4.$	
\medskip

%The only difference between Assumption \ref{AssConvergenceS}, \eqref{S} and Assumption \ref{AssConvergenceW}, \eqref{W} is in the second term, which essentially means that the rate of decay of $g$ must account for the weaker decay of the tails of the convolution kernel $K$ (see Lemma \ref{tailK}).
%
Technically, Assumption 3, \eqref{S} allows for sharper estimates of the tails of the deconvolution kernel (\ref{eq:asympkernel}) than does Assumption 4, \eqref{W}, see Lemma \ref{tailK} in Section \ref{sec:Aux}.
In this case we have to proceed differently as the approximation via a distribution free process such as $\mathbb{G}_n$ can no longer be guaranteed and we can only find a suitable Gaussian approximation depending on the standard deviation $\nu$.

Roughly speaking, we approximate $\mathbb{D}_n(x)$ in (\ref{Gaussprozess}) by the process
{\small 
	\begin{align}\label{Gaussprozesstilde}
		\widetilde{\mathbb{G}}_{n}(x) & =\frac{\sqrt{na_nh^{1+2\beta}}}{h\, \tilde{\nu}_n(x)} \sum\limits_{j}\tilde\nu_n(\omega_j)\, Z_{j} \,  K\left(\frac{w_j-x}{h};h\right),
	\end{align}
}
for a variance estimator $\tilde\nu_n$ on growing intervals $|x| \leq n\, a_n \, (1-\delta)$ for some $\delta >0$. We then replace the quantiles involving $\mathbb{G}_{n}$ in (\ref{eq:quantileapprox}), (\ref{eq:approxdiscrete}) and in (\ref{eq:confband}) by (the conditional quantiles given the sample) of $\widetilde{\mathbb{G}}_{n}$. Our theoretical developments involve a sample splitting, hence are somewhat cumbersome so that details are deferred to the Appendix, Section \ref{sec:extensionsdetails}. 

We have also simulated a version of the bootstrap for the extended model. However, as simulations show, the results are clearly not as good as for the more restrictive assumptions on $f_{\Delta}$. We have used
\begin{align*}
	f_{\Delta,1}(x)=\frac{\lambda}{2}\cdot f_{\Delta,0}(a;x-0.3)+(1-\lambda)\cdot f_{\Delta,0}(a;x)+\frac{\lambda}{2}\cdot f_{\Delta,0}(a;x+0.3),
\end{align*}
with $f_{\Delta,0}$ again the Laplace density defined in \eqref{Laplace}, $a=0.05/\sqrt{2}$ and $\lambda=0.2$.
%\begin{align*}
%f_0(x,\mu)=\frac{1}{2\lambda}\exp\left(\frac{-|x-\mu|}{\lambda}\right)
%\end{align*}
%and $\lambda=0.2$ and
% (as chosen above with $\sigma_\delta=0.05$).
%i.e. a trimodal $f$ which does not fit into the assumptions of Theorem 1.
For the signal $g_a$ in Section \ref{sec:sim} we find confidence band widths of 
$0.686$ and $0.462$ for $n=100$ and $n=750$, respectively, at simulated coverage probabilities of $6.3\%$ and $4.5\%$ and bandwidths of $0.59$ and $0.32$, for $\sigma=\sigma_{\delta}=0.1$.
\section*{Acknowledgements}
	HH gratefully acknowledges financial support form the DFG, grant Ho 3260/5-1. NB acknowledges support by the Bundesministerium f\"ur Bildung und Forschung through the project ``MED4D: Dynamic medical imaging: Modeling and analysis of
	medical data for improved diagnosis, supervision and drug
	development''. KP gratefully acknowledges financial support by the DFG through subproject A07 of CRC 755.
%\section{acknowledgements}
%

\section{Auxiliary Lemmas}\label{sec:Aux}
The proofs for the results in this section are given in Section \ref{Sec:ProofsAux}. 
\subsection{Properties of the regression function $g$,  the variance function $\nu$ and the convolution kernel $K$}  \label{sub:propnu}
Assumption \ref{AssSmoothness} stated above is basically a smoothness assumption on the function $g$. In the following lemma we list the properties of $g$ that are frequently used throughout this paper and that are implied by Assumption \ref{AssSmoothness}. 
\begin{lemma}\label{Le:g}
	Let Assumption \ref{AssSmoothness} hold.
	\begin{itemize}
		\item[(i)] The function $g$ is twice continuously differentiable.
		\item[(ii)] The function $g$ has uniformly bounded derivatives: $\|g^{(j)}\|_\infty<\infty,\, j\leq2.$
	\end{itemize}
\end{lemma}
Given Assumption \ref{AssSmoothness} (ii), the properties of the function $g$ given in Lemma \ref{Le:g} are transferred to the convolution $\gamma=g*f_{\Delta}.$ This is made precise in the following lemma.
\begin{lemma}\label{Le:gamma}
	Let Assumption \ref{AssSmoothness}  hold.
	\begin{itemize}
		\item[(i)] The function $\gamma=g*f_{\Delta}$ is twice continuously differentiable with derivatives  $\gamma^{(j)}=g^{(j)}*f_{\Delta}.$
		\item[(ii)] $\gamma\in\mathcal{W}^m(\R).$
		\item[(iii)] The function $\gamma$ has uniformly bounded derivatives: $\|\gamma^{(j)}\|_{\infty}<\infty, \, j\leq2.$
	\end{itemize}
\end{lemma}
Furthermore, the variance function $\nu^2$, defined in \eqref{VarEta}, is a function that depends on $f_{\Delta},\gamma$ and $g.$ The following lemma lists the properties of $\nu^2$, which are implied by the previous Lemmas \ref{Le:g} and \ref{Le:gamma}, and that are frequently used throughout this paper.
\begin{lemma}\label{Le:nu}
	Let Assumption \ref{AssSmoothness}  hold.
	\begin{itemize}
		\item[(i)] The variance function $\nu^2$ is uniformly bounded and bounded away from zero.
		\item[(ii)] The variance function $\nu^2$ is twice continuously differentiable with uniformly bounded derivatives.
	\end{itemize}
\end{lemma}
For the tails of the kernel, we have the following estimate. 
\begin{lemma}
	\label{tailK}
	For any $a>1$ and $x\in[0,1]$ we have
	{\small
		\begin{align*}
			\int_{\{|z|>a\}}\left(K\left(\frac{z-x}{h};h\right)\right)^2\,dz\leq \mathcal{C}\frac{2a}{a^2-x^2}\, \cdot \, \begin{cases}
				h^{-2\beta}, & \text{if Ass.~4, } \eqref{W}\text{ holds,}\\
				h^{-2\beta+2}, & \text{if Ass.~3, } \eqref{S} \text{ holds.}
			\end{cases}
		\end{align*}
	}
\end{lemma}	%
\begin{lemma}\label{Lemma:BiasVarianz}
	Let Assumptions \ref{AssSmoothness} and \ref{AssKernel} be satisfied.  Further assume that 
	%$na_n^{3/2}h^2\rightarrow\infty$ and
	$h/a_n\rightarrow0$  as $n\rightarrow\infty$. 
	
	\begin{itemize}
		\item[(i)] Then for the bias, we have that
		{\small
			\begin{align*}
				\sup_{x\in[0,1]}\bigl| \mathbb E \big[\hat g_n(x;h)\big]-g(x)\bigr|= O\left(h^{m-\frac{1}{2}}+\frac{1}{na_nh^{\beta+\frac{3}{2}}}\right)+\begin{cases}
					o\big(a_n^{s+1/2}h^{1-\beta}\big), & \text{ Ass.~3, (S),}\\
					o\big(a_n^{s+1/2}h^{-\beta}\big), & \text{Ass.~4, (W).}
				\end{cases}
			\end{align*}
		}
		\item[(ii) a)] For the variance if Assumption \ref{AssConvergenceW}, \eqref{W} holds and $na_nh^{1+\beta}\to\infty$, then we have that
		\begin{align*}
			\frac{\sigma^2 }{2C\pi}(1+O(a_n))\leq na_nh^{1+2\beta}{\rm Var}[\hat g_n(x;h)]&\leq \frac{2^{\beta}\sup_{x\in\R}\nu^2(x) }{c\pi}.
		\end{align*}		
		\item[(ii) b)]	If actually Assumption \ref{AssConvergenceS}, \eqref{S} holds  and $na_nh^{1+\beta}\to\infty$, then
		\begin{align*}
			\frac{\nu^2(x)  }{C\pi}(1+O(a_n))\leq na_nh^{1+2\beta}{\rm Var}[\hat g_n(x;h)]&\leq \frac{\nu^2(x) }{c\pi}(1+O\left(h/a_n\right)).
		\end{align*}
		Here $c,C$ and $\beta$ are the constants from  Assumption \ref{AssConvergenceW} respectively \ref{AssConvergenceS}.
	\end{itemize}
\end{lemma}

\subsection{Maxima of Gaussian processes}
Let  $\{\mathbb{X}_t\,|\,t\in T\}$ be a Gaussian process and $\rho$ be a semi-metric on $T$.  The packing number $D(T,\delta,\rho)$ is the maximum number of points in $T$ with distance $\rho$ strictly larger than $\delta>0$. Similarly to the packing numbers, the covering numbers $N(T,\delta,\rho)$ are defined as the number of closed $\rho$-balls of radius $\delta$, needed to cover $T$. Let further   $d_{\mathbb{X}}$ denote the standard deviation semi-metric on $T$, that is,
\begin{align*}
	d_{\mathbb{X}}(s,t)=\Bigl(\mathbb{E}\big[|\mathbb{X}_t-\mathbb{X}_s|^2\big]\Bigr)^{\frac{1}{2}}\quad\text{for}\quad s,t\in T.
\end{align*}
In the following, we drop the subscript if it is clear which process induces the pseudo-metric $d$.
\begin{lemma}\label{Le:Entropy} There exist  constants $C_E,C_{\widehat E}\in(0,\infty)$ such that
	\begin{itemize}
		\item[(i)] $\displaystyle N(T,\delta,d_{\mathbb{G}_n})\leq D(T,\delta,d_{\mathbb{G}_n})\leq\frac{C_E}{h^{3/2}a_n^{1/2}\delta}.$
		\item[(ii)] $\displaystyle N(T,\delta,d_{\mathbb{G}_n^{\widehat K}})\leq D(T,\delta,d_{\mathbb{G}_n^{\widehat K}})\leq\frac{C_{\widehat E}}{h^{3/2}a_n^{1/2}\delta},$ where $\mathbb{G}_n^{\widehat K}$ is defined as $\mathbb{G}_n$ with $K$ replaced by $\widehat K$, where $\widehat K(z;h)=zK(z;h)$.
	\end{itemize}
\end{lemma}

\begin{lemma}\label{Le:Comparison}
	Let $(\mathbb{X}_{n,1}(t), t\in T)$ and $(\mathbb{X}_{n,2}(t), t\in T)$ be  almost surely bounded, centered Gaussian processes on a compact index set $T$ and suppose that for any fixed $n\in\N$ $\mathrm{diam}_{d_{\mathbb{X}_{n,1}}}(T)>D_n>0.$ If 
	\begin{align*}
		d_{\mathbb{X}_{n,1}}(s,t)\leq d_{\mathbb{X}_{n,2}}(s,t)\;\forall \,s,t\in T\quad\text{and}\quad\mathbb{E}\big[\|\mathbb{X}_{n,2}\|\big]=o(1/\sqrt{\ln(n)}),
	\end{align*}
	we have that
	\begin{align*}
		\mathbb{E}\left[\|\mathbb{X}_{n,1}\|\right]\leq 2\mathbb{E}\left[\|\mathbb{X}_{n,2}\|\right]\quad\text{and hence}\quad \|\mathbb{X}_{n,1}\|=o_{\mathbb{P}}(1/\sqrt{\ln(n)}).
	\end{align*}
\end{lemma}
\begin{comment}
\subsection{Approximation results}
In this section we collected the lemmas that we use in the proofs of the approximation results of Theorem \ref{GaussAppr} (i) and Theorem \ref{Thm:MainW} (i). Some of the approximation steps are similar to those performed in previous works such as, e.g., \cite{bisdumholmun2007}. As we prove these under clearly weaker assumptions, we needed to use different techniques in our proofs, mostly based on tail inequalities for Gaussian maxima and Gaussian comparison. Therefore, we give all proves in Section \ref{Sec:ProofsAux} below.\\
Throughout this section we assume that the Assumptions of Theorem \ref{Thm:MainW} hold and only additional assumptions will be stated in the subsequent lemmas.
\end{comment}
%
\section{Proofs of Theorems \ref{GaussAppr} and \ref{Thm:Grid}}\label{Sec:Proofs}
%
%
%\subsection{Proofs of Theorems \ref{GaussAppr} and \ref{Thm:Grid}}
In the following, the letter $\mathcal{C}$ denotes a generic, positive constant, whose value may vary form line to line. The abbreviations $R_n$ and $\widetilde{R}_n$, possibly with additional subscripts, are used to denote remainder terms and their definition may vary from proof to proof. \\~\\
\begin{proof}[Proof of Theorem \ref{GaussAppr}]
We first prove assertion (i). Let $\rho_n:=n^{2/\mom}\ln(n)/\sqrt{na_nh}$ and notice that
	\begin{align*}
		&\mathbb{P}\big(\|\mathbb{D}_{n}\|\leq q_{ \|\mathbb{G}_{n}\|}(\alpha) \big)\leq \mathbb{P}\big(\|\mathbb{G}_{n}\|\leq q_{ \|\mathbb{G}_{n}\|}(\alpha)+\rho_n \big)+\mathbb{P}\big(\big|\|\mathbb{D}_{n}\|-\|\mathbb{G}_{n}\|\big|>\rho_n \big)\\
		&~~\leq\alpha+\mathbb{P}\big(q_{ \|\mathbb{G}_{n}\|}(\alpha)\leq\|\mathbb{G}_{n}\|\leq q_{ \|\mathbb{G}_{n}\|}(\alpha)+\rho_n \big)+\mathbb{P}\big(\big|\|\mathbb{D}_{n}\|-\|\mathbb{G}_{n}\|\big|>\rho_n \big),
	\end{align*}
	since the distribution of $\|\mathbb{G}_n\|$ is absolutely continuous. Analogously, it holds
	\begin{align*}
		\mathbb{P}\big(\|\mathbb{D}_{n}\|\leq q_{ \|\mathbb{G}_{n}\|}(\alpha) \big)&\geq\alpha-
		\mathbb{P}\big(q_{ \|\mathbb{G}_{n}\|}(\alpha)-\rho_n\leq\|\mathbb{G}_{n}\|\leq q_{ \|\mathbb{G}_{n}\|}(\alpha) \big)\\
		&-\mathbb{P}\big(\big|\|\mathbb{D}_{n}\|-\|\mathbb{G}_{n}\|\big|>\rho_n \big),
	\end{align*}
	and therefore
	\begin{align*}
		\left|\mathbb{P}\big(\|\mathbb{D}_{n}\|\leq q_{ \|\mathbb{G}_{n}\|}(\alpha) \big)-\alpha\right|&\leq\sup_{x\in\mathbb{R}}\mathbb{P}\left(|\|\mathbb{G}_n\|-x|\leq\rho_n\right)+\mathbb{P}\left(\big|\|\mathbb{G}_n\|-\|\mathbb{D}_n\|\big|>\rho_n\right).
	\end{align*}
	The first term on the right hand side of the inequality is the concentration function of the random variable $\|\mathbb{G}_n\|$, which can be estimated by Theorem 2.1 of \cite{CheCheKat2014}. This gives
	\begin{align*}
		\left|\mathbb{P}\big(\|\mathbb{D}_{n}\|\leq q_{ \|\mathbb{G}_{n}\|}(\alpha) \big)-\alpha\right|&\leq4\rho_n\left(\mathbb{E}[\|\mathbb{G}_n\|]+1\right)+\mathbb{P}\left(\big|\|\mathbb{G}_n\|-\|\mathbb{D}_n\|\big|>\rho_n\right).
	\end{align*}

By Lemma \ref{Le:Entropy} we have 
	$
	N([0,1],\delta,d_{\mathbb{G}_n})\leq C_E/(h^{3/2}a_n^{1/2}\delta),
	$	
	which allows to estimate the expectation $\mathbb{E}[\|\mathbb{G}_n\|]$ as follows.
	\begin{align*}
		\mathbb{E}[\|\mathbb{G}_n\|]\leq\mathcal{C}\int_{0}^{\mathrm{diam}_{d_{\mathbb{G}_n}}([0,1])}\sqrt{\ln\left(\frac{C_E}{h^{3/2}a_n^{1/2}\delta}\right)}\,d\delta\leq\mathcal{C}\sqrt{\ln(n)}.
	\end{align*}
	This yields
	\begin{align*}
		\left|\mathbb{P}\big(\|\mathbb{D}_{n}\|\leq q_{ \|\mathbb{G}_{n}\|}(\alpha) \big)-\alpha\right|\leq \mathcal{C}\sqrt{\ln(n)}\rho_n +\mathbb{P}\left(\big|\|\mathbb{G}_n\|-\|\mathbb{D}_n\|\big|>\rho_n\right).
	\end{align*}
	We now estimate the term $\mathbb{P}\left(\big|\|\mathbb{G}_n\|-\|\mathbb{D}_n\|\big|>\rho_n\right)$ in several steps. With the definition
	\begin{align}\label{eq:firstapproxprocess}
		\mathbb{G}_{n,0}(x):=\frac{h^{\beta}}{\nu(x)\sqrt{na_nh}}\sum_{j=-n}^n\nu(w_j)Z_jK\biggl(\frac{w_j-x}{h};h\biggr),
	\end{align}
	we find
	\begin{align*}
		&\mathbb{P}\left(\big|\|\mathbb{G}_n\|-\|\mathbb{D}_n\|\big|>\rho_n\right)\leq \mathbb{P}\left(\|\mathbb{G}_n-\mathbb{D}_n\|>\rho_n\right)\\
		&~~\leq\mathbb{P}\left(\|\mathbb{G}_{n,0}-\mathbb{D}_n\|>\frac{\rho_n}{2}\right)+\mathbb{P}\left(\|\mathbb{G}_n-\mathbb{G}_{n,0}\|>\frac{\rho_n}{2}\right),
	\end{align*}
	and thus
	\begin{align*}  
		&\mathbb{P}\left(\big|\|\mathbb{G}_n\|-\|\mathbb{D}_n\|\big|>\rho_n\right)\leq\mathbb{P}\left(\|\nu\mathbb{G}_{n,0}-\nu_n\mathbb{D}_n\|>\tfrac{\sigma\rho_n}{8}\right)\\
		&+\mathbb{P}\left(\left\|\tfrac{1}{\nu}-\tfrac{1}{\hat\nu}\right\|\|\nu\mathbb{G}_{n,0}\|>\tfrac{\rho_n}{4}\right)+ \mathbb{P}\left(\|\mathbb{G}_n-\mathbb{G}_{n,0}\|>\tfrac{\rho_n}{2}\right)
		=:R_{n,1}+R_{n,2}+R_{n,3}.
	\end{align*}

Consider first term $R_{n,2}$. Let $\kappa>0$ be a constant and $n$ sufficiently large such that $\kappa/\sqrt{\ln(n)}<1$. Then
	\begin{align*}
		R_{n,2}\leq\mathbb{P}\left(\left\|\tfrac{1}{\nu}-\tfrac{1}{\hat\nu}\right\|>\kappa\tfrac{\rho_n}{4\sqrt{\ln(n)}}\right)+\mathbb{P}\left(\|\nu\mathbb{G}_{n,0}\|>\tfrac{\sqrt{\ln{n}}}{\kappa}\right)=:R_{n,2,1}+R_{n,2,2}.
	\end{align*}
	The term $R_{n,2,1}$ is controlled by assumption and the term $R_{n,2,2}$ can be estimated by Borell's inequality.
	To this end, denote by $d$ the pseudo distance induced by the process $\nu\mathbb{G}_{n,0}$. It holds that
	\begin{align*}
		\mathbb{E}\Big[\sup_{x\in[0,1]}\nu(x)\mathbb{G}_{n,0}(x)\Big]\leq \mathbb{E}[\|\nu\mathbb{G}_{n,0}\|]&\leq\mathcal{C}\int_{0}^{\mathrm{diam}([0,1])}\sqrt{\ln\left(N(\delta,[0,1],d)\right)}\,d\delta\\
		&\leq\mathcal{C}\int_{0}^{\mathrm{diam}([0,1])}\sqrt{\ln\Big(\tfrac{\mathcal{C}}{h^{\frac{3}{2}}a_n^{\frac{1}{2}}\delta}\Big)}\,d\delta,
	\end{align*}
	where the last estimate follows by an application of Lemma \ref{Le:Entropy}. By a change of variables,
	using that for any $a\leq1$
	\begin{align*}
		\frac{1}{a}\int_{0}^{a}\sqrt{-\ln(x)}\,dx\leq\sqrt{-\ln(a)}+\frac{1}{\sqrt{-2\ln(a)}}\leq\mathcal{C}\sqrt{-\ln(a)},
	\end{align*}
	we obtain
	\begin{align}\label{eq:exG0}
		\mathbb{E}\Big[\sup_{x\in[0,1]}\nu(x)\mathbb{G}_{n,0}(x)\Big]\leq \mathbb{E}[\|\nu\mathbb{G}_{n,0}\|]\leq\mathcal{C}\sqrt{\ln(n)}.
	\end{align}
	Next,
	\begin{align*}
		& R_{n,2,2}\leq2 \mathbb{P}\left(\sup_{x\in[0,1]}(\nu\mathbb{G}_{n,0})(x)>\tfrac{\sqrt{\ln{n}}}{2\kappa}\right)\\
		&~=\mathbb{P}\left(\sup_{x\in[0,1]}(\nu\mathbb{G}_{n,0})(x)-\mathbb{E}\big[\sup_{x\in[0,1]}(\nu\mathbb{G}_{n,0})(x)\big]>\tfrac{\sqrt{\ln{n}}}{2\kappa}-\mathbb{E}\big[\sup_{x\in[0,1]}(\nu\mathbb{G}_{n,0})(x)\big]\right)\\
		&~\leq\mathbb{P}\left(\sup_{x\in[0,1]}(\nu\mathbb{G}_{n,0})(x)-\mathbb{E}\big[\sup_{x\in[0,1]}(\nu\mathbb{G}_{n,0})(x)\big]>\tfrac{\sqrt{\ln{n}}}{4\kappa}\right),
	\end{align*}
	for sufficiently small $\kappa$ such that $\mathbb{E}\Big[\sup_{x\in[0,1]}\nu(x)\mathbb{G}_{n,0}(x)\Big]<\tfrac{\sqrt{\ln{n}}}{4\kappa}$. An application of Borell's inequality yields
	\begin{align*}
		R_{n,2,2}\leq\exp\left( \tfrac{\ln(n)}{32\kappa^2\sigma_{[0,1]}^2}\right),
	\end{align*}
	where $\sigma_{[0,1]}^2:=\sup_{x\in[0,1]}\mathrm{Var}[\nu(x)\mathbb{G}_{n,0}(x)]$ is a bounded quantity by Lemma \ref{Le:Entropy}. For sufficiently small $\kappa$, this yields the estimate
	\begin{align*}
		R_{n,2}\leq\mathbb{P}\left(\left\|\tfrac{1}{\nu}-\tfrac{1}{\hat\nu}\right\|>\kappa\tfrac{\rho_n}{4\sqrt{\ln(n)}}\right)+\frac{\mathcal{C}}{n}.
	\end{align*}

Next, we estimate the term $R_{n,1}$, i.e., we consider the approximation of $\mathbb{D}_{n}$ by a suitable Gaussian process.
	To this end, consider the standardized random variables $\xi_j:=\eta_j/\nu(w_j)$ and write
	\begin{align}\label{Yn0+}
		\hat{\nu}_n(x)&\mathbb{D}_n(x)=  \frac{h^{\beta}}{\sqrt{na_nh}}\sum\limits_{j=-n}^n \xi_j \nu(w_j)K\left(\frac{w_j-x}{h};h\right)\notag=\frac{h^{\beta}}{\sqrt{na_nh}}\biggl[\xi_0\nu(w_0)
		K\left(-\frac{x}{h};h\right)\\
		&+  \sum\limits_{j=1}^n \xi_j\nu(w_j) K\left(\frac{w_j-x}{h};h\right)+\sum\limits_{j=-n}^{-1} \xi_j \nu(w_j)K\left(\frac{w_j-x}{h};h\right)\biggr]\notag\\
		&~=:\mathbb{D}^0_{n}(x)+\mathbb{D}_{n}^+(x)+\mathbb{D}_{n}^{-}(x),
	\end{align}
	where the processes $\mathbb{D}_{n}^+(x),$ $\mathbb{D}_{n}^{-}(x)$ and $\mathbb{D}^0_{n}(x)$ are defined in an obvious manner. 
	Define the $j$-th partial sum $S_j:=\sum_{\nu=1}^{j}\xi_{\nu}$, set $S_0\equiv0$ and write
	\begin{align*}
		&\sum_{j=1}^n\xi_j\nu(w_j)K\left(\frac{w_j-x}{h};h\right) 
		%=\sum_{j=1}^n\bigl(S_{j}-S_{j-1}\bigr)\nu(w_j)K\left(\frac{w_j-x}{h};h\right) \\
		%
		%	&~=\sum_{j=1}^nS_{j}\nu(w_j)K\left(\frac{w_j-x}{h};h\right)
		%-\sum_{j=1}^nS_{j-1}\nu(w_j)K\left(\frac{w_j-x}{h};h\right)\\
		%
		%&
		=S_n\nu(w_n)K\left(\frac{w_n-x}{h};h\right)\\
		&-\sum_{j=0}^{n-1}S_{j}\biggl[
		\nu(w_{j+1})K\left(\frac{w_{j+1}-x}{h};h\right)-\nu(w_j)K\left(\frac{w_j-x}{h};h\right)
		\biggr]~\\%.
		%\end{align*}
		%	This yields
		%	\begin{align*}
		&=S_n\nu(w_n) K\left(\frac{w_n-x}{h};h\right)
		-\sum_{j=1}^{n-1}S_{j}\int_{[w_j,w_{j+1}]}\frac{d}{dz}\Bigl(\nu(z)K\left(\frac{z-x}{h};h\right)\Bigr)\,dz. 
	\end{align*}
	By assumption, there exists a constant  $\mom>2$ such that $\mathbb{E}[|\epsilon_1|^{\mom}]<\infty$. By Lemma \ref{Le:gamma}, $\gamma$ is uniformly bounded, which implies  $\mathbb{E}[|\eta_j|^{\mom}]\leq M$ for some $M>0$ and all $j.$ 
	By Corollary 4, §5 in \cite{Sakhanenko}  there exist iid standard normally distributed random variables $Z_1,\ldots,Z_n$ such that, for $W(j):=\sum_{j=1}^nZ_j$ the following estimate holds for any positive constant $\mathcal{C}$:
	\begin{align}\label{eq:Sakhanenko}
		\mathbb{P}\left(\max_{1\leq j\leq n}|S_j-W(j)|>\tfrac{n^{2/\mom}}{2\mathcal{C}}\right)
		\leq \sum_{j=1}^{n}\mathbb{E}[|\xi_j|^{\mom}]\left(\frac{\mathcal{C}}{n^{2/\mom}}\right)^{\mom}+\mathbb{P}\left(\max_{1\leq j\leq n}\xi_j>\frac{n^{2/\mom}}{\mathcal{C}}\right).
	\end{align}
	Therefore,
	\begin{align*}
		\|\hat\nu_n\mathbb{D}_n^+-\nu\mathbb{G}_{n,0}^+\|&\leq\frac{\max_{1\leq j\leq n}|S_j-W(j)|}{\sigma\sqrt{na_nh}}\biggl[\sup_{x\in[0,1]}\nu(w_n)h^{\beta}\left|K\left(\frac{w_n-x}{h};h\right)\right|\\
		&+\sup_{x\in[0,1]}\int_{[w_1,w_{n}]}\left|\frac{d}{dz}\left(h^{\beta}\nu(z)K\left(\frac{z-x}{h};h\right)\right)\right|\,dz\biggr],
	\end{align*}
	where $\mathbb{G}_{n,0}^+$ is defined in analogy to $\mathbb{D}_{n}^+$ in \eqref{Yn0+}, with $\xi_j$ replaced by $Z_j$.
	For n sufficiently large, we have $a_n<1/2$ and thus, for $x\in[0,1]$ we have that $(w_n-x)/h\in[1/(2a_nh),1/a_nh]$ and thus
	\begin{align*}
		&\nu(w_n)h^{\beta}\left|K\left(\frac{w_n-x}{h};h\right)\right|\leq
		\nu(w_n)h^{\beta}\sup_{u>1/(2a_nh)}\left|K\left(u;h\right)\right|\\
		&\leq\nu(w_n)h^{\beta}\sup_{u>1/(2a_nh)}\left|2a_nhuK\left(u;h\right)\right|\leq2a_nh^{\beta+1}\nu(w_n)\|\cdot K\left(\cdot;h\right)\|_{\infty}\leq\mathcal{C}a_nh,
		%
		%&\leq\mathcal{C}a_nh^{\beta+1}\left\|\Phi_k(h \cdot)\frac{\hat\Phi_\gamma (\cdot )}{\Phi_{f_\Delta}(-\cdot)}\right\|_1\leq\mathcal{C}a_nh.
	\end{align*}
	by \eqref{eq:supK}.
	Next,
	\begin{align*}
		&\sup_{x\in[0,1]}\int_{[w_1,w_{n}]}\left|\frac{d}{dz}\left(h^{\beta}\nu(z)K\left(\frac{z-x}{h};h\right)\right)\right|\,dz\\
		&\leq\mathcal{C}h^{\beta}\sup_{x\in[0,1]}\int_{[w_1,w_{n}]}\left|K\left(\frac{z-x}{h};h\right)\right|+\frac{1}{h}\left|K'\left(\frac{z-x}{h};h\right)\right|\,dz\\
		&\leq\mathcal{C}h^{\beta}\sup_{x\in[0,1]}\int_{-\frac{1}{h}}^{\frac{1}{a_nh}}h\left|K\left(u;h\right)\right|+\left|K'\left(u;h\right)\right|\,dz\leq\mathcal{C}\ln(n)
	\end{align*}
	by \eqref{eq:int} in the appendix. This yields
	\begin{align*}
		\|\hat\nu_n\mathbb{D}_n^+-\nu\mathbb{G}_n^+\|\leq\|\hat\nu_n\mathbb{D}_n-\nu\mathbb{G}_n\|\leq\mathcal{C}\ln(n)\,\frac{\max_{1\leq j\leq n}|S_j-W(j)|}{\sqrt{na_nh}}.
	\end{align*}
	Hence,
	\begin{align*}
		R_{n,1}&\leq\mathbb{P}\left(\mathcal{C}\ln(n)\frac{\max_{1\leq j\leq n}|S_j-W(j)|}{\sqrt{na_nh}}>\frac{\rho_n}{2}\right)\\
		&\leq\mathbb{P}\left(\max_{1\leq j\leq n}|S_j-W(j)|>\tfrac{n^{2/\mom}}{2\mathcal{C}}\right).
	\end{align*}
	Since $0<\sigma<\nu(w_j),$ 
	$
	\mathbb{E}[|\xi_j|^{\mom}]\leq M/\sigma^{\mom}$ for all $ 1\leq j\leq n,
	$ 
	we have $R_{n,2}\leq\mathcal{C}/n$ by \eqref{eq:Sakhanenko}.

Last, we need to estimate the term $R_{n,3}$. We have
	\begin{align*}
		R_{n,3}=\mathbb{P}\left(\|\mathbb{G}_n-\mathbb{G}_{n,0}\|>\tfrac{\rho_n}{4}\right)\leq \mathbb{P}\left(\|\widetilde R_n\|>\tfrac{\rho_n}{4}\right),
	\end{align*}
	where
	\begin{align*}
		\widetilde R_n(x):=\frac{h^{\beta}}{\sqrt{na_nh}}\sum_{j=-n}^nZ_j(\nu(w_j)-\nu(x))K\left(\frac{w_j-x}{h};h\right).
	\end{align*}
	Using that by Lemma \ref{Le:nu} $|\nu(w_j)-\nu(x)|\leq\mathcal{C}|w_j-x|=h\mathcal{C}|w_j-x|/h$, we find that
	\begin{align*}
		N\left([0,1],\delta,d_{\widetilde R_n}\right)\leq\frac{\mathcal{C}}{\sqrt{a_nh}\delta}.
	\end{align*}
	Furthermore, there exist positive constants $\widehat c$ and $\widehat C$ such that
	\begin{align*}
		\widehat c\frac{h^2}{na_nh}\leq\sup_{x\in[0,1]}\mathrm{Var}\left[\widetilde R_n(x)\right]\leq \widehat C \frac{h^2}{na_nh}.
	\end{align*}
	By Theorem 4.1.2 in \cite{adltay2007}, there exists a universal constant $K$ such that, for all $u>2\sqrt{\widehat C \frac{h^2}{na_nh}}$,
	\begin{align*}
		\mathbb{P}\left(\sup_{x\in[0,1]}\widetilde R(x)\geq u\right)\leq\frac{Kun\sqrt{a_nh}}{\widehat c h^2}\cdot\Psi\left(\tfrac{u}{\sqrt{\widehat C \frac{h^2}{na_nh}}}\right),
	\end{align*}
	where $\Psi$ denotes the tail function of the standard normal distribution. Setting $u=\rho_n/8$ yields, for sufficiently large $n$,
	\begin{align*}
		&\mathbb{P}\left(\sup_{x\in[0,1]}\widetilde R(x)\geq \frac{\rho_n}{8}\right)\leq\frac{Kun\sqrt{a_nh}}{\widehat c h^2}\cdot\Psi\left(\tfrac{u}{\sqrt{\widehat C \frac{h^2}{na_nh}}}\right)\\
		&\leq\frac{Kn^{\frac{1}{2}+\frac{2}{\mom}}\ln(n)}{8\widehat ch^2}\cdot\Psi\left(\tfrac{\ln(n)n^{\frac{2}{\mom}}}{8h\sqrt{\widehat C}}\right)\leq\mathcal{C}n\exp(-n^{4/\mom}/h)\leq\frac{\mathcal{C}}{n}.
	\end{align*}
	Therefore,
	$R_{n,3}\leq\mathcal{C}/n$,
	which concludes the proof of assertion (i).

Assertion (ii) is again an immediate consequence of Lemma \ref{Lemma:BiasVarianz}.
\end{proof}

\begin{proof}[Proof of Theorem \ref{Thm:Grid}]
	On the one hand,
	\begin{align*}
		\mathbb{P}\left(\|\mathbb{D}_n\|\leq q_{\|\mathbb{G}_n\|_{\mathcal{X}_{n,m}}}(\alpha)\right)\leq\mathbb{P}\left(\|\mathbb{D}_n\|\leq q_{\|\mathbb{G}_n\|}(\alpha)\right)\leq \alpha+r_{n,1},
	\end{align*}	
	by Theorem \ref{GaussAppr}. On the other hand,
	\begin{align*}
		&\mathbb{P}\left(\|\mathbb{D}_n\|\leq q_{\|\mathbb{G}_n\|_{\mathcal{X}_{n,m}}}(\alpha)\right)\\
		&~\geq\alpha-\frac{\mathcal{C}}{n}-\mathbb{P}\left(q_{\|\mathbb{G}_n\|_{\mathcal{X}_{n,m}}}(\alpha)-\rho_n\leq\|\mathbb{G}_n\|\leq q_{\|\mathbb{G}_n\|_{\mathcal{X}_{n,m}}}(\alpha)\right).
	\end{align*}
	Note that \eqref{dist} implies
	\begin{align*}
		|s-t|\leq|\mathcal{X}_{n,m}|\quad\Rightarrow\quad d_{\mathbb{G}_n}(s,t)\leq\frac{|\mathcal{X}_{n,m}|h^{\beta}\|K^{(1)}(\cdot;h)\|_{\infty}}{h^{3/2}a_n^{1/2}}\leq\frac{\mathcal{C}|\mathcal{X}_{n,m}|}{h^{3/2}a_n^{1/2}}.
	\end{align*}
	Hence,
	\begin{align*}
		&\|\mathbb{G}_n\|_{\mathcal{X}_{n,m}}\leq\|\mathbb{G}_n\|\leq\|\mathbb{G}_n\|_{\mathcal{X}_{n,m}}+\sup_{s,t:|s-t|\leq|\mathcal{X}_{n,m}|}|\mathbb{G}_n(s)-\mathbb{G}_n(t)|\\
		&~\leq\|\mathbb{G}_n\|_{\mathcal{X}_{n,m}}+\sup_{s,t:d_{\mathbb{G}_n}(s,t)\leq\mathcal{C}\frac{|\mathcal{X}_{n,m}|}{h^{3/2}a_n^{1/2}}}|\mathbb{G}_n(s)-\mathbb{G}_n(t)|=:\|\mathbb{G}_n\|_{\mathcal{X}_{n,m}}+\tau_n.
	\end{align*}
	This yields
	\begin{align*}
		&\mathbb{P}\left(\|\mathbb{D}_n\|\leq q_{\|\mathbb{G}_n\|_{\mathcal{X}_{n,m}}}(\alpha)\right)\\
		&~\geq\alpha-\frac{\mathcal{C}}{n}-\mathbb{P}\left(q_{\|\mathbb{G}_n\|_{\mathcal{X}_{n,m}}}(\alpha)-\rho_n-\tau_n\leq\|\mathbb{G}_n\|_{\mathcal{X}_{n,m}}\leq q_{\|\mathbb{G}_n\|_{\mathcal{X}_{n,m}}}(\alpha)\right).
	\end{align*}
	By Corollary 2.2.8 in \cite{VanWel1996}  and Lemma \ref{Le:Entropy}, we find
	\begin{align*}
		\mathbb{E}[\tau_n]&\leq\mathcal{C}\int_{0}^{\mathcal{C}\frac{|\mathcal{X}_{n,m}|}{h^{3/2}a_n^{1/2}}}\sqrt{\ln\big(N([0,1],d_{\mathbb{G}_n},\eta)\big)}\,\mathrm{d}\eta\leq\mathcal{C}\tfrac{|\mathcal{X}_{n,m}|\sqrt{-\ln(|\mathcal{X}_{n,m}|)}}{h^{3/2}a_n^{1/2}}.
	\end{align*}
	Since $|\mathcal{X}_{n,m}|\leq h^{1/2}/na_n^{1/2}$, we have that $\tfrac{|\mathcal{X}_{n,m}|\sqrt{-\ln(|\mathcal{X}_{n,m}|)}}{h^{3/2}a_n^{1/2}}=o(\rho_n^2)$ and therefore, by Markov's inequality,
	\begin{align*}
		\mathbb{P}\left(\tau_n\geq\rho_n\right)\leq \mathcal{C}\tfrac{|\mathcal{X}_{n,m}|\sqrt{-\ln(|\mathcal{X}_{n,m}|)}}{h^{3/2}a_n^{1/2}}\frac{1}{\rho_n}=o(\rho_n).
	\end{align*}
	This yields
	\begin{align*}
		&\mathbb{P}\left(\|\mathbb{D}_n\|\leq q_{\|\mathbb{G}_n\|_{\mathcal{X}_{n,m}}}(\alpha)\right)\\
		&~\geq\alpha-\frac{\mathcal{C}}{n}-\mathbb{P}\left(q_{\|\mathbb{G}_n\|_{\mathcal{X}_{n,m}}}(\alpha)-2\rho_n\leq\|\mathbb{G}_n\|_{\mathcal{X}_{n,m}}\leq q_{\|\mathbb{G}_n\|_{\mathcal{X}_{n,m}}}(\alpha)\right)-o(\rho_n)\\
		&~\geq\alpha-\frac{\mathcal{C}}{n}-\sup_{x\in\mathbb{R}}\mathbb{P}\left(\|\mathbb{G}_n\|_{\mathcal{X}_{n,m}}\in [x-2\rho_n,x+2\rho_n]\right)-o(\rho_n)\geq\alpha-\mathcal{C}\left(\tfrac{1}{n}-\rho_n\right),
	\end{align*}
	where we applied Theorem 2.1 in \cite{CheCheKat2014}. Claim 1 of this theorem now follows.
 Claim 2 is an immediate consequence of Lemma \ref{Lemma:BiasVarianz}.
 \end{proof}

\section{Proofs of the auxiliary lemmas}\label{Sec:ProofsAux}

\begin{proof}[Proof of Lemma \ref{Le:g}]
Assertion (i) is a direct consequence of Sobolev's Lemma.\\
(ii) By an application of the Hausdorff-Young inequality we obtain
\begin{align*}
	\Bigl\|\frac{d^{j}}{dx^j}g\Bigr\|_{\infty}\leq\frac{1}{2\pi}\Bigl\|\Phi_{\frac{d^{j}}{dx^j}g}\Bigr\|_{1},\quad j=0,1,2.
\end{align*}
Fourier transformation converts differentiation into multiplication, that is,
\begin{align*}
	\Bigl\|\Phi_{\frac{d^{j}}{dx^j}g}\Bigr\|_{1}=\|(\cdot)^{j}\Phi_{g}\|_1,\quad j=0,1,2.
\end{align*}
Since $g\in\mathcal{W}^m(\R)$ for $m>5/2$ by Assumption \ref{AssSmoothness} it follows by an application of the Cauchy-Schwarz inequality that $\|(\cdot)^{j}\Phi_{g}\|_1<\infty$ for $j=0,1,2$ and the assertion follows.
\end{proof}
\begin{proof}[Proof of Lemma \ref{Le:gamma}]
Assertion (i) follows from Proposition 8.10 in \cite{folland1984} since $f_{\Delta}$ is a density and is hence integrable.\\
Assertion (ii) is a direct consequence of Assumption \ref{AssSmoothness} and the convolution theorem:
\begin{align*}
	\Phi_{\gamma}=\Phi_{g*f_{\Delta}(-\cdot)}=\Phi_{g}\cdot\overline{\Phi}_{f_{\Delta}},
\end{align*}
since $\Phi_{f_{\Delta}}$ is bounded.\\
Assertion (iii) follows in the same manner as the second claim of Lemma \ref{Le:g}.
\end{proof}
\begin{proof}[Proof of Lemma \ref{Le:nu}]
(i) Recall from definition \eqref{VarEta}  that
\begin{align*}
	\nu^2(z)  = \int \big(g(z+ \delta) - \gamma(z) \big)^2\, f_\Delta(\delta)\, d\, \delta + \sigma^2,\quad\sigma^2>0.
\end{align*}  
Hence, it follows from Lemma \ref{Le:g} (ii) and Lemma \ref{Le:gamma} (iii) that 	
\begin{align*}
	0<\sigma^2\leq\nu^2(z)\leq\sigma^2+2(\|g\|_{\infty}^2+\|\gamma\|_{\infty}^2)<\infty.
\end{align*}
(ii) By the first assertions of  Lemma \ref{Le:g} and Lemma \ref{Le:gamma}, the functions $g$ and $\gamma$ are twice continuously differentiable and $f_{\Delta}$ is continuous. This yields for $j=1,2$
\begin{align*}
	\frac{d^j}{dz^j}\nu^2(z)=\int\frac{\partial^j}{\partial z^j} \Big(\big(g(z+ \delta) - \gamma(z) \big)^2\, f_\Delta(\delta)\Big)\, d\, \delta.
\end{align*}
Since by Lemma \ref{Le:g} and Lemma \ref{Le:gamma} the derivatives of $g$ and $\gamma$ are uniformly bounded and $f_{\Delta}$ is a probability density, we find for $j=1,2$
\begin{align*}
	\biggl|\frac{d^j}{dz^j}\nu^2(z)\biggr|\leq\sup_{z\in\mathbb{R}}\Big|\frac{\partial^j}{\partial z^j} \big(g(z+ \delta) - \gamma(z) \big)^2\Big|.
\end{align*}
\end{proof}

\begin{proof}[Proof of Lemma \ref{tailK}]
From \eqref{eq:asympkernel}, we deduce for $w\in\R$
\begin{align*}
	i w K(w;h) &= \frac{1}{2\pi}\int_{\R}  e^{-it w} \frac{\mathrm{d}}{\mathrm{d}t}\left(\frac{\Phi_k(t)}{\Phi_{f_{\Delta}}(-t/h)}\right) dt\\
	&=\frac{1}{2\pi}\int_{\R}  e^{-it w} \left(\frac{\Phi_k^{(1)}(t)}{\Phi_{f_{\Delta}}(-t/h)}+\frac{\Phi_k(t)\cdot\Phi^{(1)}_{f_{\Delta}}(-t/h)}{h(\Phi_{f_{\Delta}}(-t/h))^2}\right) dt.
\end{align*}
Hence,
\begin{align}\label{eq:supK}
	\sup_{w\in\R}|wK(w;h)|%&\leq\frac{1}{2\pi }\left\|\frac{\Phi_k^{(1)}}{\Phi_{f_{\Delta}}(-\cdot/h)}+\frac{\Phi_k\cdot\Phi^{(1)}_{f_{\Delta}}(-\cdot/h)}{h(\Phi_{f_{\Delta}}(-\cdot/h))^2}\right\|_1\notag\\
	&\leq \frac{1}{2\pi c}\left\|\left|\Phi_k^{(1)}\right|\left\langle\tfrac{\cdot}{h}\right\rangle^{\beta}+\tfrac{\left|\Phi_k\cdot\Phi^{(1)}_{f_{\Delta}}(-\cdot/h)\right|}{ch}\left\langle\tfrac{\cdot}{h}\right\rangle^{2\beta}\right\|_1\notag\\&=\begin{cases}
		O(h^{-\beta-1}),\,\eqref{W},\\
		O(h^{-\beta}),\, \eqref{S}.
	\end{cases}
\end{align}
In particular, for all $w\in\R\backslash\{0\},$
\begin{align}\label{Eq:KL1}
	|K(w;h)|\leq\frac{\mathcal{C}}{|w|}\cdot\begin{cases}
		h^{-\beta-1}, & \eqref{W},\\
		h^{-\beta}, & \eqref{S}.
	\end{cases}
\end{align}
Now, let $a>1$. Then
\begin{align*}
	\int_{\{|z|>a\}}\left(K\left(\frac{z-x}{h};h\right)\right)^2\,dz&\leq\mathcal{C}\int_{\{|z|>a\}}\left(\frac{h}{z-x}\right)^2\,dz\cdot\begin{cases}
		h^{-2\beta-2}, & \eqref{W}\\
		h^{-2\beta}, & \eqref{S}
	\end{cases}\\
	&\leq \mathcal{C}\frac{2a}{a^2-x^2}\begin{cases}
		h^{-2\beta}, & \eqref{W}\\
		h^{-2\beta+2}, & \eqref{S}
	\end{cases}.
\end{align*}
\end{proof}
	%%%%%%%%%%%%%%%%%%%
%%%%%%%%%%%%%%%%
%%%%%%%%%%%%%%%%

\begin{proof}[Proof of Lemma \ref{Lemma:BiasVarianz}]
The proof of Lemma \ref{Lemma:BiasVarianz} is straightforward but tedious. We therefore omit the proof here and defer it to the appendix.
\end{proof}
\begin{proof}[Proof of Lemma \ref{Le:Entropy}]
	\begin{align}\label{dist}
		&d_{\mathbb{G}_n}(s,t)^2=\mathbb{E}|\mathbb{G}_n(s)-\mathbb{G}_n(t)|^2
		=\frac{h^{2\beta}}{na_nh}\sum_{j=-n}^n\Bigl|K\Bigl(\frac{w_j-s}{h};h\Bigr)-K\Bigl(\frac{w_j-t}{h};h\Bigr)\Bigr|^2\notag\\
		&\leq \frac{h^{2\beta}\|K^{(1)}(\cdot;h)\|^2_{\infty}}{a_nh}\left(\frac{s-t}{h}\right)^2\leq\mathcal{C}\frac{h^{2\beta}}{a_nh}\left\|t\frac{\Phi_k(t)}{\Phi(-t/h)}\right\|_1^2\left(\frac{s-t}{h}\right)^2,
	\end{align}
	where the last estimate follows by the Hausdorff-Young inequality and definition \eqref{eq:asympkernel}.
	Therefore, by Assumption \ref{AssConvergenceS}, there exists a constant $C_E$ such that $d_{\mathbb{G}_{n}}(s,t)\leq C_E|s-t|/(a_n^{\frac{1}{2}}h^{\frac{3}{2}})$. Now, consider the equidistant grid 
	\begin{align*}
		\mathcal{G}_{n,\delta}:=\biggl\{t_j=j\tfrac{a_n^{\frac{1}{2}}}{C_E}\delta,\,j=1,\ldots,\biggl\lfloor\tfrac{C_E}{a_n^{\frac{1}{2}}h^{\frac{3}{2}}\delta}\biggr\rfloor\biggr\}\subset[0,1]
	\end{align*}
	and note that for each $s\in[0,1]$ there exists a $t_j\in\mathcal{G}_{n,\delta}$ such that $|s-t_j|\leq a_n^{1/2}h^{3/2}\delta/(2C_E)$, which implies $d_{\mathcal{G}_n}(s,t_j)\leq\delta/2.$ Therefore, the closed $d_{\mathbb{G}_n}$-balls with centers $t_j\in \mathcal{G}_{n,\delta}$ and radius $\delta/2$ cover the space $[0,1]$, i.e.,
	\begin{align*}
		N([0,1],\delta/2,d_{\mathbb{G}_n})\leq\tfrac{C_E}{h^{\frac{3}{2}}a_n^{\frac{1}{2}}\delta}.
	\end{align*}
	The relationship $N([0,1],\delta,d_{\mathbb{G}_n})\leq D([0,1],\delta,d_{\mathbb{G}_n})\leq N([0,1],\delta/2,d_{\mathbb{G}_n})$
	now yields the first claim of the lemma. Using that, by Assumption \ref{AssConvergenceS},
	\begin{align*}
		\|\widehat K^{(1)}(\cdot,h)\|_{\infty}&\leq\|\cdot K^{(1)}(\cdot,h)\|_{\infty}+\|K(\cdot,h)\|_{\infty}\\
		&\leq \mathcal{C}\left(\left\|\frac{\mathrm{d}}{\mathrm{d}t}\left(\frac{t\Phi_k(t)}{\Phi_{f_{\Delta}}(-t/h)}\right)\right\|_1+\left\|\left(\frac{t\Phi_k(t)}{\Phi_{f_{\Delta}}(-t/h)}\right)\right\|_1\right)\leq\mathcal{C}h^{-\beta},
	\end{align*}

the second claim follows along the lines of the first claim.
\end{proof}

\newpage

\begin{appendix}
	
\section{Appendix: Proofs of technical results in the main paper} \label{sec:appendix}

	\subsection{Proof of Lemma \ref{Lemma:BiasVarianz}}
	
	(i) We have that
	\begin{align*}
		\mathbb{E}\left[\hat g_n(x;h)\right]&=\frac{1}{na_nh}\sum_{j=-n}^{n}\gamma(w_j)K\left(\frac{w_j-x}{h};h\right)\\
		&=\frac{1}{h}\sum_{j=-n}^{n}\int_{w_j}^{w_j+\frac{1}{na_n}}\gamma(w_j)K\left(\frac{w_j-x}{h};h\right)\,dz\\
		&=\int_{-\frac{1}{a_n}}^{\frac{1}{a_n}+\frac{1}{na_n}}\gamma(z)K\left(\frac{z-x}{h};h\right)\,dz+R_{n,1}(x)+R_{n,2}(x),
	\end{align*}	  
	where 
	\begin{align*}
		R_{n,1}(x)=\frac{1}{h}\sum_{j=-n}^{n}\int_{w_j}^{w_j+\frac{1}{na_n}}\frac{\rm{d}}{{\rm d}u}\left(\gamma(u)K\left(\frac{u-x}{h};h\right)\right)\bigg|_{u=z}(w_j-z)\,dz
	\end{align*}
	and
	\begin{align*}
		R_{n,2}(x)=\frac{1}{2\, h}\sum_{j=-n}^{n}\int_{w_j}^{w_j+\frac{1}{na_n}}\frac{\rm{d}^2}{{\rm d}u^2}\left(\gamma(u)K\left(\frac{u-x}{h};h\right)\right)\bigg|_{u=\tilde w_j(z)}(w_j-z)^2\,dz.
	\end{align*}
	Then,
	\begin{align*}
		R_{n,1}(x)&=\frac{1}{h}\sum_{j=-n}^{n}\int_{w_j}^{w_j+\frac{1}{na_n}}\gamma'(z)K\left(\frac{z-x}{h};h\right)(w_j-z)\,dz
		\\&-\frac{1}{h^2}\sum_{j=-n}^{n}\int_{w_j}^{w_j+\frac{1}{na_n}}\gamma(z)K'\left(\frac{z-x}{h};h\right)(w_j-z)\,dz\\&=:R_{n,1,1}(x)+R_{n,1,2}(x).
	\end{align*}	
	Now,
	{\small
		\begin{align*}
			na_nh|R_{n,1,1}(x)|\leq\int_{-\frac{1}{a_n}}^{\frac{1}{a_n}}\left|\gamma'(z)K\left(\frac{z-x}{h};h\right)\right|\,dz\leq\|\gamma'\|_2\left\|K\left(\frac{\cdot-x}{h};h\right)\right\|_2=O\left(\frac{1}{h^{-\frac{1}{2}+\beta}}\right).
		\end{align*}
	}
	%By Assumption \ref{AssSmoothness} (iii) $\F\gamma\in\mathcal{W}^s,s>\frac{1}{2}$, therefore 
	Analogously,
	\begin{align*}
		|R_{n,1,2}(x)|=O\left(\frac{1}{na_nh^{\frac{3}{2}+\beta}}\right).
	\end{align*}
	Furthermore
	\begin{align*}
		&|R_{n,2}(x)|\leq\biggl[\frac{1}{n^2a_n^2h}\sum_{j=-n}^{n}\int_{w_j}^{w_j+\frac{1}{na_n}}\left|\gamma''(\tilde w_j(z))K\left(\frac{\tilde w_j(z)-x}{h};h\right)\right|\,dz
		\\&+\frac{2}{n^2a_n^2h^2}\sum_{j=-n}^{n}\int_{w_j}^{w_j+\frac{1}{na_n}}\left|\gamma'(\tilde w_j(z))K'\left(\frac{\tilde w_j(z)-x}{h};h\right)\right|\,dz\biggr]\\
		&+\frac{1}{n^2a_n^2h^3}\sum_{j=-n}^{n}\int_{w_j}^{w_j+\frac{1}{na_n}}\left|\gamma(\tilde w_j(z))K''\left(\frac{\tilde w_j(z)-x}{h};h\right)\right|\,dz\\
		&=:[R_{n,2,1}(x)]+R_{n,2,2}(x).
	\end{align*}
	Then,
	$
	R_{n,2,1}(x)=O\left(\frac{1}{n^2a_n^{3}h^{2+\beta}}+\frac{1}{n^2a_n^{3}h^{2+\beta}}\right)=o\left(\frac{1}{na_nh^{\frac{3}{2}+\beta}}\right),
	$
	since $h/a_n\to0$. By Assumption 1 (iii) $\F\gamma\in\mathcal{W}^s,s>\frac{1}{2}$, therefore $\gamma\in L^1(\R)$ and hence
	\begin{align*}
		R_{n,2,2}&\leq\frac{\mathcal{C}}{h^{3+\beta}(na_n)^2}\left(\|\gamma\|_1+\frac{1}{na_n^2}\right)=O\left(\frac{1}{n^2a_n^2h^{3+\beta}}+\frac{1}{n^3a_n^4h^{3+\beta}}\right)\\
		&=O\left(\frac{1}{na_nh^{\frac{3}{2}+\beta}}\left(\frac{1}{na_nh^{\frac{3}{2}}}+\frac{1}{n^2a_n^3h^{\frac{3}{2}+\beta}}\right)\right)=o\left(\frac{1}{na_nh^{\frac{3}{2}+\beta}}\right),
	\end{align*}
	where we used again that $h/a_n\to0$.
	Hence, in total we find
	\begin{align*}
		\mathbb{E}[\hat g_n(x;h)]&=\frac{1}{h}\int_{-\frac{1}{a_n}}^{\frac{1}{a_n}}\gamma(z)K\Bigl(\frac{z-x}{h};h\Bigr)\,dz+O\Bigl(\frac{1}{na_nh^{\frac{3}{2}+\beta}}\Bigr).
	\end{align*}
	Next, we enlarge the domain of integration and estimate the remainder as follows. By the Cauchy-Schwarz inequality we obtain 
	\begin{align*}
		\int_{|z|>\frac{1}{a_n}}\gamma(z)K\left(\frac{z-x}{h};h\right)\,dz&\leq\left(\int_{|z|>\frac{1}{a_n}}|\gamma(z)|^2\,dz\right)^{\frac{1}{2}}\\&\times\left(\int_{|z|>\frac{1}{a_n}}\left|K\left(\frac{z-x}{h};h\right)\right|^2\,dz\right)^{\frac{1}{2}}.
	\end{align*}
	By Assumption 1 (iii)
	\begin{align*}
		\int_{|z|>\frac{1}{a_n}}|\gamma(z)|^2\,dz\leq\int_{|z|>\frac{1}{a_n}}\frac{(1+z^2)^s}{(1+\frac{1}{a_n^2})^{s}}|\gamma(z)|^2\,dz\leq\mathcal{C}a_n^{2s}.
	\end{align*}	
	By Lemma 4
	\begin{align*}
		\int_{\{|z|>\frac{1}{a_n}\}}\left(K\left(\frac{z-x}{h};h\right)\right)^2\,dz&\leq \mathcal{C}a_n\left.\begin{cases}
			h^{-2\beta}, & (W)\\
			h^{-2\beta+2}, & (S)
		\end{cases}\right\}=O(a_n\|K(\cdot;h)\|_2^2).
	\end{align*}
	Hence,
	\begin{align*}
		\mathbb{E}[\hat g_n(x;h)]=\frac{1}{h}\int\gamma(z)K\Bigl(\frac{z-x}{h};h\Bigr)\,dz&+o\biggl(\frac{1}{na_nh^{\frac{3}{2}+\beta}}\biggr)\\&+\begin{cases}
			O\left(a_n^{s+1/2}h^{1-\beta}\right),& \text{Ass.~3, (S),}\\
			O\left(a_n^{s+1/2}h^{-\beta}\right),& \text{ Ass.~4, (W).}
		\end{cases}
	\end{align*}
	Furthermore, by Plancherel's equality and the convolution theorem,
	\begin{align*}
		&\frac{1}{h}\int\gamma(z)K\Bigl(\frac{z-x}{h};h\Bigr)\,dz=\frac{1}{h}
		\int\gamma(z)\widetilde K\Bigl(\frac{z-x}{h};h\Bigr)\,dz\\&=
		\frac{1}{2\pi h }\int\overline{\Phi_{\gamma}(z)}\Phi_{\widetilde K(\frac{\cdot-x}{h};h)}(z)\,dz
		=\frac{1}{2\pi  }\int\exp(ixz)\overline{\Phi_{\gamma}(z)}\Phi_{\widetilde K(\cdot;h)}(zh)\,dz\\
		&=\frac{1}{2\pi  }\int\exp(ixz)\overline{\Phi_{f_{\Delta}}(z)\Phi_g(z)}\frac{\Phi_k(hz)}{\Phi_{f_{\Delta}}(-z)}\,dz=\frac{1}{2\pi  }\int\exp(ixz)\Phi_g(-z)\Phi_k(hz)\,dz.
	\end{align*}
	Hence,
	\begin{align*}
		&\frac{1}{h}\int\gamma(z)K\Bigl(\frac{z-x}{h};h\Bigr)\,dz\\&=
		\frac{1}{2\pi  }\int\exp(ixz)\Phi_g(-z)\,dz+\frac{1}{2\pi  }\int\exp(ixz)\Phi_g(-z)\left(\Phi_k(hz)-1\right)\,dz\\
		&=g(x)+\frac{1}{2\pi  }\int\exp(ixz)\Phi_g(z)\left(\Phi_k(hz)-1\right)\,dz=g(x)+R_n(x).
	\end{align*}	
	Finally,
	\begin{align*}
		|R_{n}(x)|\leq \mathcal{C}\int_{|z|>D/h}|\Phi_g(z)|\,dz&\leq\left(\int_{|z|>D/h}\left(\frac{1}{1+z^2}\right)^m\,dz\right)^{\frac{1}{2}}\\&\times\left(\int_{|z|>D/h}|\langle z\rangle^{2m}|\Phi_g(z)|^2\,dz\right)^{\frac{1}{2}},
	\end{align*}
	which yields the estimate $R_{n}(x)=O(h^{m-\frac{1}{2}})$.\\
	(ii)  In the situation of both, (ii)a) and (ii)b), we have
	\begin{align*}
		{\rm Var}\bigl[\hat g_n(x;h)\bigr]&=\frac{1}{n^2a_n^2h^2}\sum_{j=-n}^{n}\nu^2(w_j)\biggl|K\Bigl(\frac{w_j-x}{h};h\Bigr)\biggr|^2\\
		&=\frac{1}{na_nh^2}\int_{-\frac{1}{a_n}}^{\frac{1}{a_n}+\frac{1}{na_n}}\nu^2(z)\biggl(K\Bigl(\frac{z-x}{h};h\Bigr)\biggr)^2\,dz
		+R_{n}(x),%\\
		%
		%&=\frac{1}{na_nh^2}\biggl(\int_{\R}\nu^2(z)\bigl(K^{(n)}\bigr)^2\Bigl(\frac{z-x}{h}\Bigr)\,dz
		%+O\left(\frac{1}{na_n}\right)+o\left(a_n^{2\beta+\frac{3}{4}}\right)\biggr),
	\end{align*}
	where
	{\small
		\begin{align*}
			R_{n}(x)&=\frac{1}{na_nh^2}\sum_{j=-n}^n\int_{w_j}^{w_{j}+\frac{1}{na_n}}\biggl[\nu^2(w_j)\biggl(K\Bigl(\frac{w_j-x}{h};h\Bigr)\biggr)^2-\nu^2(z)\biggl(K\Bigl(\frac{z-x}{h};h\Bigr)\biggr)^2\biggr]\,dz.\\
		\end{align*}
	}
	Then
	{\small
		\begin{align*}
			R_{n}(x)&=\frac{1}{na_nh^2}\sum_{j=-n}^n\int_{w_j}^{w_{j}+\frac{1}{na_n}}\nu^2(w_j)\biggl[\biggl(K\Bigl(\frac{w_j-x}{h};h\Bigr)\biggr)^2-\biggl(K\Bigl(\frac{z-x}{h};h\Bigr)\biggr)^2\biggr]\,dz\\
			&+\frac{1}{na_nh^2}\sum_{j=-n}^n\int_{w_j}^{w_{j}+\frac{1}{na_n}}\biggl(K\Bigl(\frac{z-x}{h};h\Bigr)\biggr)^2\bigl[\nu^2(w_j)-\nu^2(z)\bigr]\,dz\\
			&=:R_{n,1}(x)+R_{n,2}(x).
		\end{align*}
	}
	By uniform Lipschitz continuity of $\nu^2$ (see Lemma 3 (ii)), it is immediate that
	\begin{align*}
		|R_{n,2}(x)|&\leq\frac{\mathcal{C}}{n^2a_n^2h^2}\int\biggl(K\Bigl(\frac{z-x}{h};h\Bigr)\biggr)^2\leq\frac{\mathcal{C}}{n^2a_n^2h}\|K(\cdot,h)\|^2\\&=O\left(\frac{1}{n^2a_n^2h^{1+2\beta}}\right).
	\end{align*}
	Next, we consider the term $R_{n,1}$ for which we will use a Taylor expansion of $K^2(\cdot;h)$. To this end, notice first from (7) that for any $l$
	\begin{align*}
		K^{(l)}(w;h)=\frac{(-1)^l}{2\pi}\int\frac{e^{-itw}\Phi_k(t)\cdot t^{l}}{\Phi_{f_{\Delta}(-t/h)}}\,dt,
	\end{align*}	
	where the functions $F_l:t\mapsto\Phi_k(t)\cdot t^{l}$	is uniformly bounded by $1$ and twice continuously differentiable by Assumption 2 for any $l\in\mathbb{N}$. It follows that $K^2(\cdot;h)$ is smooth with integrable derivatives of all orders $l\in\mathbb{N}$, since 
	\begin{align*}
		(K^2)^{(l)}(w;h)=(K\cdot K)^{(l)}(w;h)=\sum_{k=0}^l\binom{l}{k}K^{(l-k)}(w;h)K^{(k)}(w;h),
	\end{align*}
	by the general Leibniz rule. This yields
	\begin{align}\label{eq:ii}
		\int\left|(K^2)^{(l)}(w;h)\right|\,dw&\leq\sum_{k=0}^l\binom{l}{k}\left(\int |K^{(l-k)}(w;h)|^2\,dw\right)^{\frac{1}{2}}\notag\\&\quad\times\left( \int |K^{(k)}(w;h)|^2\,dw\right)^{\frac{1}{2}}\leq\frac{\mathcal{C}}{h^{2\beta}},\tag{S1}
	\end{align}
	by Lemma 4 and the previous discussion. Let $M\in\mathbb{N}$ be such that $M\geq\frac{2}{\beta}-1$. It follows that
	\begin{align*}
		|R_{n,1}(x)|&\leq\frac{\mathcal{C}}{na_nh^{2}}\Biggl(\sum_{j=-n}^n\int_{w_j}^{w_{j}+\frac{1}{na_n}}\left[\sum_{l=1}^M\biggl|(K^2)^{(l)}\Bigl(\frac{z-x}{h};h\Bigr)\biggr|\left(\frac{1}{na_nh}\right)^l\right]\,dz\\ &\qquad\qquad+\left(\frac{1}{na_nh}\right)^{M+1}\cdot\frac{1}{a_nh^{2\beta}}\Biggr).
	\end{align*}
	By \eqref{eq:ii}, we deduce
	\begin{align*}
		|R_{n,1}(x)|&\leq\frac{\mathcal{C}}{na_nh^{1+2\beta}}\left(\frac{1}{na_nh}+\frac{1}{ha_n}\left(\frac{1}{na_nh}\right)^{M+1}\right)\\&\leq\frac{\mathcal{C}}{na_nh^{1+2\beta}}\left(\frac{1}{na_nh}+\left(\frac{1}{na_nh^{1+\frac{2}{M+1}}}\right)^{M+1}\right).
	\end{align*}
	Since $M\geq\frac{2}{\beta}-1$, we finally obtain
	\begin{align*}
		|R_{n,1}(x)|&=o\left(\frac{1}{na_nh^{1+2\beta}}\right).
	\end{align*}
	An application of Plancherel's theorem and Assumption 5 give
	\begin{align}\label{eq:iib}
		\frac{1}{\pi C h^{2\beta}}\leq\|K(\cdot;h)\|^2_2=\frac{1}{2\pi}\bigg\|\frac{\Phi_k}{\Phi_{f_{\Delta}}(\cdot/h)}\bigg\|_2^2\leq \frac{1}{\pi c}\bigg(1+\frac{1}{h^2}\bigg)^{\beta}.\tag{S2}
	\end{align}	
	Now, if (S) holds, an application of Lemma 4 yields
	\begin{align*}
		\sup_{x\in[0,1]}\left|\int_{\frac{1}{a_n}}^{\infty}\biggl(K\Bigl(\frac{z-x}{h};h\Bigr)\biggr)^2\,dz\right|=O\left(\frac{a_nh^2}{h^{2\beta}}\right)=O\left(a_nh^2\|K(\cdot;h)\|_2^2\right).
	\end{align*}
	Thus,
	\begin{align*}
		&\int_{-\frac{1}{a_n}}^{\frac{1}{a_n}+\frac{1}{na_n}}\nu^2(z){\small\biggl(K\Bigl(\frac{z-x}{h};h\Bigr)\biggr)^2}\,dz=h\nu^2(x)\|K(\cdot;h)\|^2_2(1+O(ha_n))+R_{n,3}(x),
	\end{align*}	
	where
	\begin{align*}
		|R_{n,3}(x)|&:=\bigg|\int_{-\frac{1}{a_n}}^{\frac{1}{a_n}+\frac{1}{na_n}}\big(\nu^2(z)-\nu^2(x)\big)\biggl(K\Bigl(\frac{z-x}{h};h\Bigr)\biggr)^2\,dz\bigg|\\
		&\leq\mathcal{C}h\int_{-\frac{2}{a_nh}}^{\frac{2}{a_nh}}\big|\nu^2(zh-x)-\nu^2(x)\big|\bigl(K(z;h)\bigr)^2\,dz\\
		&\leq\mathcal{C}h\int_{-\frac{2}{a_nh}}^{\frac{2}{a_nh}}\big|zh\big|\bigl(K(z;h)\bigr)^2\,dz.
	\end{align*}
	By (24) we have $|z\cdot K\bigl(z;h\bigr)|\leq\mathcal{C}/h^{\beta}$ and
	\begin{align*}
		|R_{n,3}(x)|\leq\mathcal{C}h^{2-\beta}\int_{-\frac{2}{a_nh}}^{\frac{2}{a_nh}}\big|K(z;h)\big|\,dz=O\left(\frac{h^2\ln(n)}{h^{2\beta}}\right),
	\end{align*}	
	since, by (24) and (25) in the proof of Lemma 4 and (S),
	\begin{align}\label{eq:int}
		&\int_{-\frac{2}{a_nh}}^{\frac{2}{a_nh}}\bigl|K\bigl(z;h\bigr)\bigr|\,dz\leq\frac{\mathcal{C}}{h^{\beta}}+ \int_{1\leq|z|\leq\frac{2}{a_nh}}\bigl|K\bigl(z;h\bigr)\bigr|\,dz\notag\\
		&\leq\frac{\mathcal{C}}{h^{\beta}}\left(1+ \int_{1\leq|z|\leq\frac{2}{a_nh}}\frac{1}{|z|}\,dz\right)\leq\frac{\mathcal{C}}{h^{\beta}}\left(1+\ln(2/a_nh)\right)=O(\ln(n)/h^{\beta}).\tag{S3}
	\end{align}	
	Assertion (ii)b) now follows. \\
	(ii)a) Under (W)
	\begin{align*}
		&\int_{-\frac{1}{a_n}}^{\frac{1}{a_n}+\frac{1}{na_n}}\nu^2(z)\biggl(K\Bigl(\frac{z-x}{h};h\Bigr)\biggr)^2\,dz\leq\sup_{y\in\mathbb{R}}\nu^2(y)h\int|K(z;h)|^2\,dz,
	\end{align*}	
	and the second inequality of (ii)a) follows by \eqref{eq:iib}.
	Furthermore,
	\begin{align*}
		\int_{-\frac{1}{a_n}}^{\frac{1}{a_n}+\frac{1}{na_n}}\nu^2(z)\biggl(K\Bigl(\frac{z-x}{h};h\Bigr)\biggr)^2\,dz&\geq\sigma^2\int_{-\frac{1}{a_n}}^{\frac{1}{a_n}+\frac{1}{na_n}}\biggl(K\Bigl(\frac{z-x}{h};h\Bigr)\biggr)^2\,dz\\&\geq\frac{h\sigma^2}{2}\|K(\cdot;h)\|_2^2, 
	\end{align*}
	for sufficiently large $n$ by Lemma 4.
	Now, the first inequality of (ii)a) follows by \eqref{eq:iib}, which concludes the proof of this lemma.
	\qed
	
	\subsection{Extensions to non-equidistant design}\label{sec:nonequidist}
	For ease of notation, we considered an equally spaced design in the main document. However, this can be relaxed to more general designs. In this section, we restate the main results (Theorem 1, Theorem 2, Lemma 5) and adjust their proofs to the case where the design is generated by a known positive design density $f_{D,n}$ on $[0,\infty)$ as follows
	\begin{align*}
		\frac{j}{n+1}=\int_{0}^{w_j}f_{D,n}(z)\,dz,\quad j=1,\ldots,n,
	\end{align*}
	and $w_j=-w_{-j}$. Note that, given the latter definition, we have $f_{D,n}(z)=(n+1)/na_nI_{[0,1/a_n]}(z)$. Furthermore, we require the following regularity assumptions.
	\begin{assumption}\label{ass:design}
		~\\
		\vspace{-0.4cm}
		\begin{enumerate}
			\item The density $f_{D,n}$ is continuously differentiable, $f_{D,n}\in C^1(\mathrm{supp}(f_{D,n}))$.
			\item There exist constants $c_D$ and $C_D$ such that $c_Da_n\leq f_{D,n}\leq C_Da_n$.
			\item The derivative $f_{D,n}'$ is uniformly bounded, $|f_{D,n}'|\leq a_nC_{D'}$.
		\end{enumerate}
	\end{assumption}	
	Regarding our estimator, we need to make the following adjustment to accommodate the more general design 
	\[
	\hat g_n(x;h) = \frac{1}{nh } \sum\limits_{j=-n}^n \frac{Y_j}{f_{D,n}(w_j)} K\left(\frac{w_j-x}{h};h\right).
	\]
	This yields the following adjusted Lemma 5 and adjusted proof.
	
	\begin{lemma}
		\label{Lemma:BiasVarianzD}
		Let Assumptions 1 and 2 be satisfied.  Further assume that 
		%$na_n^{3/2}h^2\rightarrow\infty$ and
		$h/a_n\rightarrow0$  as $n\rightarrow\infty$. 
		
		\begin{itemize}
			\item[(i)\;\;\;] Then for bias, we have that
			{\small
				\begin{align*}
					\sup_{x\in[0,1]}\bigl|\mathbb E\big[\hat g_n(x;h)\big]-g(x)\bigr|= O\left(h^{m-\frac{1}{2}}+\frac{1}{na_nh^{\beta+\frac{3}{2}}}\right)+\begin{cases}
						o\big(a_n^{s+1/2}h^{1-\beta}\big), & \text{ Ass.~3, (S),}\\
						o\big(a_n^{s+1/2}h^{-\beta}\big), & \text{Ass.~4, (W).}
					\end{cases}
				\end{align*}
			}
			\item[(ii) a)] For the variance if Assumption 5, (W) holds and $na_nh^{1+\beta}\to\infty$, then we have that
			\begin{align*}
				\frac{1}{C_D}\cdot	\frac{\sigma^2 }{2C\pi}(1+O(a_n))\leq na_nh^{1+2\beta}{\rm Var}[\hat g_n(x;h)]&\leq \frac{2^{\beta}\sup_{x\in\R}\nu^2(x) }{c\pi}\cdot\frac{1}{c_D}.
			\end{align*}		
			\item[(ii) b)]	If actually Assumption 3, (S) holds  and $na_nh^{1+\beta}\to\infty$, then
			\begin{align*}
				\frac{1}{f_{D,n}(x)}\cdot\frac{\nu^2(x)  }{C\pi}(1+O(a_n))\leq na_nh^{1+2\beta}{\rm Var}[\hat g_n(x;h)]&\leq \frac{\nu^2(x) }{c\pi}(1+O\left(h/a_n\right))\cdot\frac{1}{f_{D,n}(x)},
			\end{align*}

			Here $c,C$ and $\beta$ are the constants from  Assumption 5 respectively 3.
		\end{itemize}
	\end{lemma}	
	\begin{proof}[Proof of Lemma \ref{Lemma:BiasVarianzD}]
		(i) We have that
		\begin{align*}
			\mathbb{E}\left[\hat g_n(x;h)\right]&=\frac{1}{nh}\sum_{j=-n}^{n}\frac{\gamma(w_j)}{f_{D,n}(w_j)}K\left(\frac{w_j-x}{h};h\right)\\
			&=\frac{1}{nh}\sum_{j=-n}^{n-1}\int_{w_j}^{w_{j+1}}\frac{\gamma(w_j)}{f_{D,n}(w_j)(w_{j+1}-w_j)}K\left(\frac{w_j-x}{h};h\right)\,dz.
		\end{align*}
		Next, observe
		\begin{align*}
			f_{D,n}(w_j)(w_{j+1}-w_j)&= F_{D,n}(w_{j+1})-F_{D,n}(w_j)-\frac{1}{2}f_{D,n}'(w_j^*)(w_{j+1}-w_j)^2,\\
			&=\frac{1}{n+1}+\frac{1}{2}f_{D,n}'(w_j^*)(w_{j+1}-w_j)^2,
		\end{align*}	
		where $F_{D,n}$ is the primitive of $f_{D,n}$. This yields
		\begin{align*}
			\left|f_{D,n}(w_j)(w_{j+1}-w_j)-\frac{1}{n}\right|\leq\frac{1}{n^2}+\frac{1}{2}C_{D'}\frac{1}{n^2a_n^2c_D}
		\end{align*}
		and thus
		\begin{align}\label{eq:rep}
			\frac{1}{f_{D,n}(w_j)(w_{j+1}-w_j)}=n\left(1+O\left(\tfrac{1}{na_n^2}\right) \right).\tag{S4}
		\end{align}
		Replacement of $1/f_{D,n}(w_j)(w_{j+1}-w_j)$ by the latter estimate now yields
		\begin{align*}
			\mathbb{E}\left[\hat g_n(x;h)\right]
			&=\frac{1}{h}\int_{-\frac{1}{a_n}}^{\frac{1}{a_n}}\gamma(z)K\left(\frac{z-x}{h};h\right)\,dz\left(1+O\left(\tfrac{1}{na_n^2}\right) \right)+R_{n,1}(x)+R_{n,2}(x),
		\end{align*}	  
		where 
		\begin{align*}
			R_{n,1}(x)=\frac{1}{h}\sum_{j=-n}^{n-1}\int_{w_j}^{w_{j+1}}\frac{\rm{d}}{{\rm d}u}\left(\gamma(u)K\left(\frac{u-x}{h};h\right)\right)\bigg|_{u=z}(w_j-z)\,dz
		\end{align*}
		and
		\begin{align*}
			R_{n,2}(x)=\frac{1}{2\, h}\sum_{j=-n}^{n-1}\int_{w_j}^{w_{j+1}}\frac{\rm{d}^2}{{\rm d}u^2}\left(\gamma(u)K\left(\frac{u-x}{h};h\right)\right)\bigg|_{u=\tilde w_j(z)}(w_j-z)^2\,dz.
		\end{align*}
		For $z\in [w_j,w_{j+1}]$  we obtain the following estimates by Assumption \ref{ass:design}:
		$$|w_j-z|\leq|w_j-w_{j+1}|\leq1/(na_nc_D)$$. Therefore, the rest of the proof of claim (i) follows along the lines of the proof of Lemma 5 (i).\\
		(ii)  In the situation of both, (ii)a) and (ii)b), we have
		\begin{align*}
			{\rm Var}\bigl[\hat g_n(x;h)\bigr]&=\frac{1}{n^2h^2}\sum_{j=-n}^{n}\frac{\nu^2(w_j)}{f_{D,n}(w_j)^2}\biggl|K\Bigl(\frac{w_j-x}{h};h\Bigr)\biggr|^2\\
			&=\frac{1}{nh^2}\int_{-\frac{1}{a_n}}^{\frac{1}{a_n}}\frac{\nu^2(z)}{f_{D,n}(z)}\biggl(K\Bigl(\frac{z-x}{h};h\Bigr)\biggr)^2\,dz\cdot\left(1+O\left(\frac{1}{na_n^2}\right) \right)
			+R_{n}(x),%\\
			%
			%&=\frac{1}{na_nh^2}\biggl(\int_{\R}\nu^2(z)\bigl(K^{(n)}\bigr)^2\Bigl(\frac{z-x}{h}\Bigr)\,dz
			%+O\left(\frac{1}{na_n}\right)+o\left(a_n^{2\beta+\frac{3}{4}}\right)\biggr),
		\end{align*}
		where
		{\small
			\begin{align*}
				R_{n}(x)&=\frac{1}{nh^2}\sum_{j=-n}^{n-1}\int_{w_j}^{w_{j+1}}\biggl[\frac{\nu^2(w_j)}{nf_{D,n}(w_j)^2(w_{j+1}-w_j)}\biggl(K\Bigl(\frac{w_j-x}{h};h\Bigr)\biggr)^2-\frac{\nu^2(z)}{f_{D,n}(z)}\biggl(K\Bigl(\frac{z-x}{h};h\Bigr)\biggr)^2\biggr]\,dz.\\
			\end{align*}
		}
		Then
		{\small
			\begin{align*}
				R_{n}(x)&=\frac{1}{nh^2}\sum_{j=-n}^{n-1}\int_{w_j}^{w_{j+1}}\frac{\nu^2(w_j)}{nf_{D,n}(w_j)^2(w_{j+1}-w_j)}\biggl[\biggl(K\Bigl(\frac{w_j-x}{h};h\Bigr)\biggr)^2-\biggl(K\Bigl(\frac{z-x}{h};h\Bigr)\biggr)^2\biggr]\,dz\\
				&+\frac{1}{nh^2}\sum_{j=-n}^n\int_{w_j}^{w_{j+1}}\biggl(K\Bigl(\frac{z-x}{h};h\Bigr)\biggr)^2\biggl[\frac{\nu^2(w_j)}{nf_{D,n}(w_j)^2(w_{j+1}-w_j)}-\frac{\nu^2(z)}{f_{D,n}(z)}\biggr]\,dz.
				%	&=:R_{n,1}(x)+R_{n,2}(x).
			\end{align*}
		}
		Using \eqref{eq:rep}, we further obtain
		\begin{align*}
			R_{n}(x)&=\frac{1}{nh^2}\sum_{j=-n}^{n-1}\int_{w_j}^{w_{j+1}}\frac{\nu^2(w_j)}{f_{D,n}(w_j)}\biggl[\biggl(K\Bigl(\frac{w_j-x}{h};h\Bigr)\biggr)^2-\biggl(K\Bigl(\frac{z-x}{h};h\Bigr)\biggr)^2\biggr]\,dz\cdot\left(1+O\left(\tfrac{1}{na_n}\right)\right)\\
			&+\frac{1}{nh^2}\sum_{j=-n}^{n-1}\int_{w_j}^{w_{j+1}}\biggl(K\Bigl(\frac{z-x}{h};h\Bigr)\biggr)^2\biggl[\frac{\nu^2(w_j)}{f_{D,n}(w_j)}-\frac{\nu^2(z)}{f_{D,n}(z)}\biggr]\,dz\\
			&+\frac{\mathcal{C}}{n^2a_nh^2}\sum_{j=-n}^{n-1}\int_{w_j}^{w_{j+1}}\biggl(K\Bigl(\frac{z-x}{h};h\Bigr)\biggr)^2\frac{\nu^2(w_j)}{f_{D,n}(w_j)}\,dz=:R_{n,1}(x)+R_{n,2}(x)+R_{n,3}(x).
		\end{align*}
		It holds that
		\begin{align*}
			|R_{n,3}(x)|\leq\frac{\mathcal{C}}{n^2a_n^2h}\|K(\cdot,h)\|_2^2=O\left(\frac{1}{n^2a_nh^{2\beta+1}}\right).
		\end{align*}
		Furthermore,
		\begin{align*}
			R_{n,2}(x)&=\frac{1}{nh^2}\sum_{j=-n}^{n-1}\int_{w_j}^{w_{j+1}}\biggl(K\Bigl(\frac{z-x}{h};h\Bigr)\biggr)^2\nu^2(z)\frac{f_{D,n}(z)-f_{D,n}(w_j)}{f_{D,n}(z)f_{D,n}(w_j)}\,dz\\
			&+\frac{1}{nh^2}\sum_{j=-n}^{n-1}\int_{w_j}^{w_{j+1}}\biggl(K\Bigl(\frac{z-x}{h};h\Bigr)\biggr)^2\biggl[\frac{\nu^2(w_j)-\nu^2(z)}{f_{D,n}(w_j)}\biggr]\,dz%=:R_{n,2,1}(x)+R_{n,2,2}(x)
		\end{align*}
		By uniform Lipschitz continuity of $\nu^2$ (see Lemma 3 (ii)) and $f_{D,n}$ (by Assumption \ref{ass:design}), it is immediate that
		\begin{align*}
			|R_{n,2}(x)|&\leq\frac{\mathcal{C}}{n^2a_n^2h^2}\int\biggl(K\Bigl(\frac{z-x}{h};h\Bigr)\biggr)^2\leq\frac{\mathcal{C}}{n^2a_n^2h}\|K(\cdot,h)\|^2\\&=O\left(\frac{1}{n^2a_n^2h^{1+2\beta}}\right).
		\end{align*}
		Using again that $|w_j-z|\leq|w_j-w_{j+1}|\leq1/(na_nc_D)$, the rest of the proof of claim (ii) follows along the lines of the proof of Lemma 5 (ii).
	\end{proof}	
	\subsection{Role of the hyperparameters}\label{sec:hyperrole}
	In this section, we discuss the setting presented in Example 1 of the main document in more detail, in order to shed some light into the role of the parameters $a_n$ and $h$, as well as the assumptions made for our theoretical considerations. 
	In particular, we show that, in a typical setting, the conditions listed in Assumption 4 are satisfied if $h$ is the rate optimal bandwidth. As an example, we consider the case of a function $g\in\mathcal{W}^m(\R), m>5/2,$ of bounded support, $[-1,1]$, say, $f_{\Delta}$ as in Definition (8) in the main document with $a=1,$ i.e.,
	\begin{align*}
		f_{\Delta}(x)=\tfrac{1}{2}e^{-|x|}\quad\text{with}\quad\Phi_{ f_{\Delta}}(t)=\langle t\rangle^{-2},
	\end{align*}
	and $\mathbb{E}[\epsilon_1^4]\leq\infty$, i.e., $M=4$. Here, the parameter $\beta$, which gives the degree of ill-posedness of the problem and which is defined in Assumption 3 in the main document is given by $\beta=2$.
	In this example, Assumption 4 (i) becomes 
	\begin{align*}
		\frac{\ln(n)}{n^{-\frac{1}{2}}a_nh}+\frac{h}{a_n}+\ln(n)^2h+\ln(n)a_n+\frac{1}{na_nh^{5}}=o(1).
	\end{align*}
	Given that $h/a_n=o(1)$, the last term asymptotically dominates the first one, such that  Assumption 4 (i) reduces to
	\begin{align}\label{eq:4i}
		\frac{h}{a_n}+\ln(n)^2h+\ln(n)a_n+\frac{1}{na_nh^{5}}=o(1).\tag{S5}
	\end{align}
	In density deconvolution, where the target density $g$ is contained in a H\"older ball such that $|g^{(r)}(x)-g^{(r)}(x+ \delta)|\leq B\delta^{a}$ for some radius $B>0$, some positive integer $r$ and some $a\in[0,1)$, the rate optimal bandwidth is of order $n^{-1/(2(r+a+\beta)-1)}$. According to our Assumption 1 (i), we have $g\in\mathcal{W}^m$. Due to the embedding $\mathcal{W}^m(\mathbb{R})\subset C^{r+a}(\R)$, we can replace $r+a$ in the above bandwidth by $m-1/2$ in terms of our parametrization of the smoothness of the target function. This yields $h\simeq n^{-1/(2(m+\beta))}$. With this choice of $h$, given that $1/(na_nh^{5})=o(1)$ by Assumption 4 (i), Assumption 4 (ii) becomes
	\begin{align}\label{eq:4ii}
		%\sqrt{na_nh^{2m+2\beta}}+\sqrt{na_n^{2s+1}h^2}+1/\sqrt{na_nh^2}=
		\sqrt{a_n}+\sqrt{n^{(m+\beta-1)/(m+\beta)}a_n^{2s+1}}=o(1/\sqrt{\ln(n)}).\tag{S6}
	\end{align}
	Next, we will consider the design parameter $a_n$.This parameter ensures that asymptotically, observations on the whole real line are available. This is necessary since the function $\gamma$ will typically be of unbounded support, even if the function $g$ itself is of bounded support as it is the case in this example. Condition \eqref{eq:4ii} can only be satisfied if $a_n=n^{-\varepsilon}$, where $\epsilon$ can only be as small as the parameter $s$ (see Assumption 1 (iii)) allows.  Assumption 1 (iii) in the main document is an assumption on the decay of $\gamma$. As a rule of thumb it can be said that this assumption is met if both functions $g$ and $f_{\Delta}$ decay sufficiently fast. To give some more intuition, we now provide some computations for our specific example, for which the assumption is met for any $s$. We find
	\begin{align*}
		\int\langle z\rangle^{2s}(\gamma(z))^2\,dz&=\int\langle z\rangle^{2s}\left(\int g(t)f_{\Delta}(t-z)\,dt\right)^2\,dz\\
		&=\frac{1}{4}\int\langle z\rangle^{2s}\left(\int g(t)\exp(-|t-z|)\,dt\right)^2\,dz.
	\end{align*}
	Since $g$ is supported on $[-1,1]$ by assumption, we now split the outer integral into integrals over different regions, such that all values of  $t$ with a contribution to the inner integral will be either larger or smaller. The remaining term is bounded by a constant.
	\begin{align*}
		\int\langle z\rangle^{2s}(\gamma(z))^2\,dz&\leq\mathcal{C}+\frac{1}{4}\int_{|z|>1}\langle z\rangle^{2s}\left(\int_{|t|\leq1} g(t)\exp(-|t-z|)\,dt\right)^2\,dz.
	\end{align*}
	If $z>1$ in the outer integral, $\exp(-|t-z|)=\exp(-z+t)$ and thus
	\begin{align*}
		\int_{z>1}\exp(-z)\langle z\rangle^{2s}\left(\int_{|t|\leq1} g(t)\exp(t)\,dt\right)^2\,dz<\infty,
	\end{align*}
	for any $s>0$. Analogously,
	\begin{align*}
		\int_{z<-1}\exp(z)\langle z\rangle^{2s}\left(\int_{|t|\leq1} g(t)\exp(-t)\,dt\right)^2\,dz<\infty,
	\end{align*}
	for any $s>0$. Therefore, in the setting of this example, Assumption 1 (iii) holds for any $s$,such that $a_n$ can be chosen of order $n^{-\varepsilon}$ for $\varepsilon$ arbitrarily small. In this case, the remaining conditions in \eqref{eq:4i} and \eqref{eq:4ii} become $1/(na_nh^5)=o(1)$ and $\sqrt{\ln(n)a_n}=o(1)$, respectively. Using $a_n=n^{-\varepsilon}$ and $h=n^{-1/(2(m+\beta))}\geq n^{-1/9}$, we find that Assumption 4 holds for any $\varepsilon<4/9$.
	
	\subsection{Extensions: Details}\label{sec:extensionsdetails}
	Our theoretical developments for the procedure in Section 5 in the main document actually involve a sample splitting. To this end, let $(d_n)_{n\in\N}$ be a sequence of natural numbers, $d_n\to\infty$ and $d_n=o(n)$, $1/d_n=o(1/\ln(n)^2)$ and let $\mathcal{J}_n:=\{-n,\ldots,n\} \backslash\{-n+k\cdot d_n\,|\,1\leq k\leq 2n/d_n\}$, i.e., we remove each $d_n$-th data point from our original sample. Now, set $\mathcal{Y}_{1}:=\{Y_j\,|\,j\in\mathcal{J}_n\}$ as well as $\mathcal{Y}_{2}:=\{Y_j\,|\,j\in\{-n,\ldots,n\}\backslash\mathcal{J}_n\}$. This way, the asymptotic properties of the estimator based on the main part of the sample, $\mathcal{Y}_1$, remain the same. Let further, for $j\in\mathcal{J}_n$,  $\vartriangle_j$ denote the difference of $\omega_j$ and its left neighbor, that is, $\vartriangle_j=1/(na_n)$ if $j-1\in\mathcal{J}_n$ and      $\vartriangle_j=2/(na_n)$ else. Define the estimator $\tilde g_n$, based on $\mathcal{Y}_1$ by
	\begin{align*}
		\tilde g_n(x;h)=\frac{1}{h}\sum_{j\in\mathcal{J}_n}Y_j\vartriangle_jK\left(\frac{w_j-x}{h};h\right).
	\end{align*} We now formulate an analog of Theorem 1 under Assumption 5.
	\begin{theorem}\label{Thm:MainW}
		Let Assumptions  \ref{AssKernel}, 4 (i) and  5  be satisfied.  Let further $\tilde{\nu}_n$ be a nonparametric estimator of the standard deviation in model \eqref{modeleta} based on $\mathcal{Y}_2$ such that for some sequence of positive numbers $b_n\to0$ for which $a_n/b_n=o(1/\ln(n))$ we have that 
		\begin{align}\label{eq:vest}
			\mathbb{E}\left[ \sup_{|j|\leq b_nn}|\tilde\nu_n(\omega_j)-\nu(\omega_j)|\right]=o\big(1/(\ln(n)\ln\ln(n))\big)\quad\text{and}\quad\tilde{\nu}_n>\tilde{\sigma}>0\tag{S7}
		\end{align}
		for some constant $\tilde{\sigma}>0$. 	
		
		\begin{enumerate}
			\item There exists a sequence of independent standard normally distributed random variables $(Z_n)_{n\in\Z}$, independent of $\tilde{\nu}_n$, such that for
				\begin{align}\label{Gaussprozesstildetilde}
					\widetilde{\mathbb{D}}_n(x) & =\frac{\sqrt{na_nh^{1+2\beta}}}{\tilde{\nu}(x)}\bigl(\tilde g_n(x;h)-\mathbb{E}[\tilde g_n(x;h)]\bigr),\notag\\
					\widetilde{\mathbb{G}}_{n}(x) & =\frac{\sqrt{na_nh^{1+2\beta}}}{h\, \tilde{\nu}_n(x)} \sum\limits_{j\in\mathcal{J}_n,|j|\leq n b_n}\tilde\nu_n(\omega_j)\vartriangle_j Z_{j} K\left(\frac{w_j-x}{h};h\right),\tag{S8}
				\end{align}
			we have that
			\[ \forall\ \alpha \in (0,1):\lim_{n \to \infty} \mathbb{P}\big(\|\widetilde{\mathbb{D}}_{n}\|\leq q_{ \|\widetilde{\mathbb{G}}_{n}\|}(\alpha) \big)= \alpha,\]
			where $q_{ \|\widetilde{\mathbb{G}}_{n}\|}(\alpha)$ denotes the $\alpha$-quantile of $ \|\widetilde{\mathbb{G}}_{n}\|$.	
			\item If, in addition,  $\sqrt{na_nh^{2m+2\beta}}+\sqrt{na_n^{2s+1}}+1/\sqrt{na_nh^2} =o(1/\sqrt{\ln(n)})$.
			and if Assumption \ref{AssSmoothness} is satisfied, $\mathbb{E}[\tilde g_n(x;h)]$ in (11) can be replaced by $g(x).$ 			    
		\end{enumerate}   
	\end{theorem}
	%
	%Under the assumptions of Theorem \ref{Thm:MainW}, we have
	%\[ \forall\ \alpha \in (0,1):\quad \lim_{n \to \infty} \mathbb{P}\big(\|\widetilde{\mathbb{D}}_{n}\|\leq q_{ %\|\widetilde{\mathbb{G}}_{n}\|_{\mathcal{X}_{n,m}}}(\alpha) \big)= \alpha.\]
	\begin{proof}[Proof of Theorem \ref{Thm:MainW}]
		%The structure of this proof is similar to that of the proof of Theorem \ref{GaussAppr}.
		We require that 
		\begin{align}
			\|\widetilde{\mathbb{D}}_n- \widetilde{\mathbb{G}}_n\|=o_{\mathbb{P}}(1/\sqrt{\ln(n)}),\tag{Step 1}
		\end{align}	
		as well as
		\begin{align}
			\mathbb{E}\|\widetilde{\mathbb{G}}_n\|=O_{\mathbb{P}}(\sqrt{\ln(n)}).\tag{Step 2}
		\end{align}	
		\paragraph{Step 1 a: Gaussian Approximation}~\\
		Lemmas \ref{Le:tilde1} and \ref{Le:14b} are in preparation for the Gaussian approximation where the target process $\widetilde{\mathbb{G}}_n$ is first approximated  by processes $\widetilde{\mathbb{G}}_{n,0}^{b_n}$ and $\widetilde{\mathbb{G}}_{n,0}$.% in order to mimic the situation of Lemma \ref{Le:Gaussappr} which will finally be applied to prove Lemma \ref{Le:tilde2}.
		\begin{lemma}\label{Le:tilde1}
			We shall show that 	
			\begin{align}\label{eq:unnorm}
				\|\nu\, \widetilde{\mathbb{G}}_{n,0}^{b_n}-\tilde\nu_n\,\widetilde{\mathbb{G}}_n\|=o_{\mathbb{P}}(1/\sqrt{\ln(n)}),
			\end{align}
			where
			\begin{align*}
				\widetilde{\mathbb{G}}_{n,0}^{b_n}:=\frac{\sqrt{na_nh^{1+2\beta}}}{h\nu(x)} \sum\limits_{j\in\mathcal{J}_n,|j|\leq n b_n}\nu(\omega_j)\vartriangle_j Z_{j} K\left(\frac{w_j-x}{h};h\right).
			\end{align*}
		\end{lemma}
		\begin{proof}
			%
			%
			%Writing
			%
			%	\begin{align*}
			%	\widetilde{\mathbb{G}}_{n,0}^{b_n}-\widetilde{\mathbb{G}}_n = \frac{\nu\widetilde{\mathbb{G}}_{n,0}^{b_n}-\tilde\nu_n\, \widetilde{\mathbb{G}}_n}{\nu}  - \frac{\nu-\tilde\nu_n}{\nu}\, \widetilde{\mathbb{G}}_n,
			%	\end{align*}
			%
			%	 the statement of the lemma follows since $\nu$ is uniformly bounded from below, see (\ref{VarEta}), and $\|\tilde \nu_n-\nu\|=o_{\mathbb{P}}(1/\ln(n))$ as well as Lemma \ref{Le:tilde3} imply that the second term is  $o_{\mathbb{P}}(1/\sqrt{\ln(n)})$ as well. 
			Let $\mathcal{Z}:=\{Z_{-n},\ldots,Z_{n}\}$ be iid standard Gaussian random variables as in the proof of Theorem 1, and $c_n:=1/\ln(n)\ln\ln(n)$.
			\begin{align*}
				&\mathbb{P}_{\mathcal{Z},\mathcal{Y}_2}\bigl(\|\widetilde{\mathbb{G}}_{n,0}^{b_n}-\widetilde{\mathbb{G}}_{n}\|>\delta/\sqrt{\ln(n)}\bigr)\leq \mathbb{P}_{\mathcal{Z},\mathcal{Y}_2}\biggl(\sup_{|j|\leq b_n n}|\nu(w_j)-\tilde\nu_n(w_j)|>c_n\biggr)\\
				&+\mathbb{P}_{\mathcal{Z},\mathcal{Y}_2}\biggl(\|\widetilde{\mathbb{G}}_{n,0}^{b_n}-\widetilde{\mathbb{G}}_{n}\|>\delta/\sqrt{\ln(n)}\,;\,\sup_{|j|\leq b_n n}|\nu(w_j)-\tilde\nu_n(w_j)|\leq c_n\biggr)=P_{n,1}+P_{n,2}.
			\end{align*}
			By assumption, $P_{n,1}=o(1)$.
			\begin{align*}
				&	P_{n,2}\leq \mathbb{P}_{\mathcal{Z},\mathcal{Y}_2}\biggl(\|\widetilde{\mathbb{G}}_{n,0}^{b_n}-\widetilde{\mathbb{G}}_{n}\|I_{\{\sup_{|j|\leq b_n n}|\nu(w_j)-\tilde\nu_n(w_j)|\leq c_n\}}>\delta/\sqrt{\ln(n)}\biggr)\\
				&=\mathbb{E}_{\mathcal{Y}_2}\left[\mathbb{P}_{\mathcal{Z},\mathcal{Y}_2}\biggl(\|\widetilde{\mathbb{G}}_{n,0}^{b_n}-\widetilde{\mathbb{G}}_{n}\|I_{\{\sup_{|j|\leq b_n n}|\nu(w_j)-\tilde\nu_n(w_j)|\leq c_n\}}>\delta/\sqrt{\ln(n)}\,\bigg|\,\mathcal{Y}_2\biggr)\right]\\
				&\leq\mathbb{E}_{\mathcal{Y}_2}\left[\mathbb{E}_{\mathcal{Z},\mathcal{Y}_2}\biggl(\|\widetilde{\mathbb{G}}_{n,0}^{b_n}-\widetilde{\mathbb{G}}_{n}\|I_{\{\sup_{|j|\leq b_n n}|\nu(w_j)-\tilde\nu_n(w_j)|\leq c_n\}}\,\bigg|\,\mathcal{Y}_2\biggr)\right]\cdot\frac{\sqrt{\ln(n)}}{\delta},
			\end{align*}
			by the conditional Markov inequality. Set
			\begin{align*}
				R_n(x):=\big(\widetilde{\mathbb{G}}_{n,0}^{b_n}(x)-\widetilde{\mathbb{G}}_{n}(x)\big)I_{\{\sup_{|j|\leq b_n n}|\nu(w_j)-\tilde\nu_n(w_j)|\leq c_n\}}
			\end{align*}
			and 
			\begin{align*}
				\widetilde R_n(x):=\frac{2h^{\beta}c_n}{\sqrt{na_nh}}\sum_{|j|\leq b_nn}Z_jK\biggl(\frac{w_j-x}{h};h\biggr).
			\end{align*}
			Since $d_{R_n}(s,t)\leq d_{\widetilde{R}_{n}}(s,t)$ for all $s,t\in[0,1]$ and for all samples $\mathcal{Y}_2$ and $\mathbb{E}\big[\|\widetilde{R}_n\|\,|\,\mathcal{Y}_2\big]=\mathbb{E}\|\widetilde{R}_n\|=O(c_n\sqrt{\ln(n)})=o(1/\sqrt{\ln(n)})$, it follows by Lemma 7 that $\mathbb{E}_{\mathcal{Z},\mathcal{Y}_2}[\|R_n\|\,|\,\mathcal{Y}_2]\leq\mathbb{E}[\|\widetilde R_{n}\|]$ for all samples. Therefore, $P_{n,2}=o(1/\sqrt{\ln(n)})$ and the claim follows.
		\end{proof}
		\begin{lemma}\label{Le:14b}
			We have that 
			\begin{align*}
				\|\nu\widetilde{\mathbb{G}}_{n,0}^{b_n}-\nu\widetilde{\mathbb{G}}_{n,0}\|=o_{\mathbb{P}}(1/\sqrt{\ln(n)}).
			\end{align*}
			where 	
			\begin{align}\label{eq:firstapproxprocess}
				\widetilde{\mathbb{G}}_{n,0}(x):=\frac{h^{\beta}}{\nu(x)\sqrt{na_nh}}\sum_{\substack {j\in \mathcal{J}_n\\nb_n<|j|\leq n}}\nu(w_j)\Delta_jZ_jK\biggl(\frac{w_j-x}{h};h\biggr).\tag{S9}
			\end{align}
		\end{lemma}
		\begin{proof}
			Let
			\begin{align*}
				R_n(x)&:=\nu\widetilde{\mathbb{G}}_{n,0}^{b_n}-\nu\widetilde{\mathbb{G}}_{n,0}=\frac{\sqrt{na_nh^{1+2\beta}}}{h}\sum_{\substack {j\in \mathcal{J}_n\\nb_n<|j|\leq n}}\nu(w_j)\Delta_jZ_jK\left(\frac{w_j-x}{h};h\right).
			\end{align*}
			Then, since $\Delta_j^2\leq4/(n^2a_n^2)$ and $\mathcal{J}_n\subset\{-n,\ldots,n\},$
			\begin{align*}
				{\rm Var}[R_n(x)]%&=\frac{na_nh^{1+2\beta}}{h^2}\sum_{\substack {j\in \mathcal{J}_n\\nb_n<|j|\leq n}}\nu^2(w_j)\Delta_j^2K^2\left(\frac{w_j-x}{h};h\right)\\
				&\leq\frac{4h^{2\beta}}{na_nh}\sum_{nb_n<|j|\leq n}\nu^2(w_j)K^2\left(\frac{w_j-x}{h};h\right).
			\end{align*}
			From the proof of Lemma \ref{Lemma:BiasVarianz} we deduce
			\begin{align*}
				{\rm Var}[R_n(x)]
				%&\leq\mathcal{C}h^{2\beta-1}\int_{\frac{b_n}{a_n}\leq|z|\leq \frac{1}{a_n}}K^2\left(\frac{z-x}{h};h\right)\,dz+o\left(\frac{1}{na_nh}\right)\\
				&\leq\mathcal{C}\left(\frac{ha_n}{b_n}+\frac{1}{na_nh}\right)=o\left(\frac{1}{\ln(n)^2}\right).
			\end{align*}
			An application of the following Lemma \ref{Le:SmallVariance} concludes this proof.
			\end{proof}
			\begin{lemma}\label{Le:SmallVariance}
				Let $(\mathbb{X}_n(t), t\in T)$ be an almost surely bounded Gaussian process on a compact index set $T$ with $\sigma^2_{T,n}:=\|\mathrm{Var}(\mathbb{X}_n)\|=o(1/\ln(n)^2)$ such that $N(T,\delta,d_{\mathbb{X}_n})\leq(n/\delta)^{a}$ for all $\delta\leq\sigma_{T,n}$ and some $a\in(0,\infty)$. Then
				\begin{align*}
					\|\mathbb{X}_n\|=o_{\mathbb{P}}(1/\sqrt{\ln(n)}).
				\end{align*}
			\end{lemma}
%			\noindent\textbf{Proof of Lemma \ref{Le:SmallVariance}.}
\begin{proof}
			Fix $\delta>0$. An application of Theorem 4.1.2 in \cite{adltay2007} yields, for large enough $n\in\N$ and a universal constant $K$
			\begin{align*}
				&\mathbb{P}\biggl(\|\mathbb{X}_n\|_T>\tfrac{\delta}{\sqrt{\ln(n)}}\biggr)%\leq2\mathbb{P}\biggl(\sup_{t\in T}\mathbb{X}_n(t)>\tfrac{\delta}{2\sqrt{\ln(n)}}\biggr)\\
				\leq2\biggl(\tfrac{Kn\delta}{2\sqrt{\ln(n)}\sigma_{T,n}^2}\biggr)^a\cdot\tfrac{2\sqrt{\ln(n)}\sigma_{T,n}}{\sqrt{2\pi}\delta}\cdot\exp\bigg(-\tfrac{\delta^2}{8\ln(n)\sigma^2_{T,n}}\bigg).
			\end{align*}
			Now
			\begin{align*}
				\mathbb{P}\biggl(\|\mathbb{X}_n\|_T>\tfrac{\delta}{\sqrt{\ln(n)}}\biggr)%&\leq\mathcal{C}n^{a+1}\delta^{a-1}\sigma_{n,T}^{-2a+1}\exp\left(-\tfrac{\delta^2}{o(1)\sigma_{T,n}}\right)\\
				&\leq \mathcal{C}n^{a+1}\delta^{a-1}\sigma_{n,T}^{-2a+1}\exp\left(-\sigma_{T,n}^{-1}\right)\left(\exp\left(-\sigma_{T,n}^{-1}\right)\right)^{\frac{\delta^2}{o(1)}-1}.
			\end{align*}
			Since $\sigma_{n,T}^{-2a+1}\exp\left(-\sigma_{T,n}^{-1}\right)\to0$ as $\sigma_{n,T}\to0$ and $\left(\exp\left(-\sigma_{T,n}^{-1}\right)\right)^{\frac{\delta^2}{o(1)}-1}=o(n^{-b})$ for any fixed $b\in(0,\infty)$, the claim of the lemma now follows.
		\end{proof}
		\begin{lemma}\label{Le:tilde2}
			We have that
			\begin{align*}
				\|\widetilde{\mathbb{G}}_{n}-\widetilde{\mathbb{D}}_n\|=o_{\mathbb{P}}(1/\sqrt{\ln(n)}).\end{align*}
		\end{lemma}
		\begin{proof}
			Since by assumption (\ref{eq:vest}), $\tilde\nu_n$ is bounded away from zero, it suffices to show that
			\begin{align*}
				\|\tilde\nu_n\widetilde{\mathbb{G}}_{n}-\tilde\nu_n\widetilde{\mathbb{D}}_n\|=o_{\mathbb{P}}(1/\sqrt{\ln(n)}).
			\end{align*}
			We estimate
			\begin{align*}
				\|\tilde\nu_n\, \widetilde{\mathbb{G}}_{n}-\tilde\nu_n\, \widetilde{\mathbb{D}}_n\|&\leq\, \|\tilde\nu_n\, \widetilde{\mathbb{G}}_{n}-\nu\widetilde{\mathbb{G}}_{n,0}^{b_n}\|+ \|\nu\widetilde{\mathbb{G}}_{n,0}^{b_n}-\nu\widetilde{\mathbb{G}}_{n,0}\|+\|\nu\widetilde{\mathbb{G}}_{n,0}-\tilde\nu_n\, \widetilde{\mathbb{D}}_n\|\\
				&=o_{\mathbb{P}}\left(1/\sqrt{\ln(n)}\right) + \|\nu\widetilde{\mathbb{G}}_{n,0}-\tilde\nu_n\, \widetilde{\mathbb{D}}_n\|
			\end{align*}
			The claim now follows along the lines of the Gaussian approximation in the proof of Theorem 1.
		\end{proof}
		\paragraph{Step 2: Expectation of the maximum}
		\begin{lemma}\label{Le:tilde3}
			\begin{align*}
				\mathbb{E}\|\widetilde{\mathbb{G}}_{n}\|=O(\sqrt{\ln(n)}).
			\end{align*}
		\end{lemma}
		\begin{proof}
			Write $\mathbb{E}[\|\widetilde{\mathbb{G}}_{n}\|]=\mathbb{E}\big[\mathbb{E}[\|\widetilde{\mathbb{G}}_{n}\|]\,\big|\,\mathcal{Y}_2\big]$ and define
			\begin{align*}
				\mathbb{X}_n(t):=\frac{2M_nh^{\beta-1/2}}{\sqrt{na_n}\tilde{\sigma}}\sum_{j=-n}^nZ_j\nu(w_j)
				K\left(\frac{w_j-t}{h};h\right),
			\end{align*}
			where
			$
			M_n:=\sqrt{\max_{|j|\leq b_nn}|\tilde{\nu}^2(w_j)-\nu^2(w_j)|\frac{1}{\sigma^2}}+1.
			$
			Conditionally on $\mathcal{Y}_2$, $(\widetilde{\mathbb{G}}_n(t),t\in[0,1])$ is a Gaussian process and we find for $s,t\in[0,1]$ and for all possible samples $\mathcal{Y}_2$ the following set of inequalities hold
			{\small
				\begin{align*}
					&\mathbb{E}\big[\big|\widetilde{\mathbb{G}}_{n}(s)-\widetilde{\mathbb{G}}_{n}(t)\big|^2\,\big|\,\mathcal{Y}_2\big]\leq\frac{na_nh^{2\beta-1}}{\tilde{\sigma}^2}\sum_{|j|\leq nb_n}\tilde{\nu}^2(w_j)\Delta_j^2
					\left|K\left(\frac{w_j-s}{h};h\right)-K\left(\frac{w_j-t}{h};h\right)\right|^2\\
					&\leq\frac{4h^{2\beta-1}}{na_n\tilde{\sigma}^2}\sum_{|j|\leq nb_n}\bigg[|\tilde{\nu}^2(w_j)-\nu^2(w_j)|\frac{\nu^2(w_j)}{\sigma^2}+\nu^2(w_j)\bigg]
					\left|K\left(\frac{w_j-s}{h};h\right)-K\left(\frac{w_j-t}{h};h\right)\right|^2\\
					&\leq\mathbb{E}\big[\big|\mathbb{X}_{n}(s)-\mathbb{X}_{n}(t)\big|^2\,\big|\,\mathcal{Y}_2\big].
				\end{align*}
			}
			An application of Lemma 7 yields, for all samples,
			$
			\mathbb{E}\big[\|\widetilde{\mathbb{ G}}_{n}\|\,\big|\,\mathcal{Y}_2\big]\leq\mathbb{E}\big[\|X_{n}\|\,\big|\,\mathcal{Y}_2\big].
			$
			Therefore,
			$
			\mathbb{E}[\|\widetilde{\mathbb{G}}_{n}\|]\leq
			\mathbb{E}\Big[\mathbb{E}[\|X_{n}\|]\,\big|\,\mathcal{Y}_2\big]\Big]\leq\mathcal{C}\mathbb{E}\|\mathbb{G}_n\|\cdot\mathbb{E}M_n.
			$
			
			An application of Jensen's inequality and \eqref{eq:vest} yield  $\mathbb{E}M_n\leq2$ for sufficiently large $n\in\N.$
		\end{proof}
		%
		\begin{comment}
		Next, we estimate $\|\widetilde{\mathbb{G}}_{n,0}^{b_n}-\widetilde{\mathbb{G}}_{n,0}\|.$
		From Lemma \ref{Le:FinalApprW} we find
		\begin{align*}
		\widetilde{\mathbb{G}}_{n,0}^{b_n}(x)-\widetilde{\mathbb{G}}_{n,0}(x)=\frac{h^{\beta}}{\sqrt{h}}\int_{\frac{b_n}{a_n}}^{\frac{1}{a_n}}\nu(z)K\Bigl(\frac{z-x}{h};h\Bigr)\,dW_z+o\bigg(\frac{1}{\sqrt{\ln(n)}}\bigg)=:\widetilde R_{n,0}(x)+o\bigg(\frac{1}{\sqrt{\ln(n)}}\bigg).
		\end{align*}
		We have
		\begin{align*}
		\mathrm{Var}(\widetilde R_{n,0}(x))\leq \mathcal{C}h^{2\beta}\int_{\{|z|>b_n/a_nh\}}K^{2}(z;h)\,dz\leq \mathcal{C}h^{2\beta}\frac{a_n^2h^2}{b_n^2}\int\big(zK(z;h)\big)^2\,dz=O\Bigl( \tfrac{a_n^2}{b_n^2}\Bigr),
		\end{align*}
		by Assumption \ref{AssConvergenceW}.  Since $a_n/b_n=o(1/\sqrt{\ln(n)})$, by Lemma \ref{Le:SmallVariance} we obtain $\|R_n\|=o(1/\sqrt{\ln(n)}).$
		\end{comment}
		\paragraph{Step 3: Anti-Concentration}~\\
		Following the arguments given in the proof of Theorem 1 concludes the proof of this theorem.
	\end{proof}
	
\end{appendix}	

\end{document}